\documentclass[11pt]{amsart}

\usepackage{amscd}
\usepackage{amsthm}
\usepackage[centertags]{amsmath}
\usepackage{amsfonts}
\usepackage{newlfont}
\usepackage{graphicx}
\usepackage[usenames]{color}
\usepackage{amsfonts, amssymb}
\usepackage{mathrsfs}
\usepackage{latexsym}
\usepackage{bbold}
\usepackage{tikz}
\usepackage{verbatim}
\usepackage{tabulary,tabularx}
\usepackage[all]{xy}
\usepackage[latin1]{inputenc}

\newtheorem{theorem}{Theorem}[section]

\newtheorem{proposition}[theorem]{Proposition}
\newtheorem{corollary}[theorem]{Corollary}
\newtheorem{lemma}[theorem]{Lemma}
\newtheorem{definition}[theorem]{Definition}

\theoremstyle{definition}
\newtheorem{remark}[theorem]{Remark}
\newtheorem{example}[theorem]{Example}

\theoremstyle{definition}

\theoremstyle{remark}

\newcommand{\LL}{\mathcal L}
\newcommand{\PP}{\mathcal P}

\newcommand{\zN}{\mathbb N}

\renewcommand{\H}{H(\mathbb{C})}
\newcommand{\C}{\mathbb C}

\newcommand{\sub}{\subseteq}
\newcommand{\zT}{\mathbb T}
\newcommand{\zD}{\mathbb D}

\newcommand{\f}{\frac}

\renewcommand{\l}{\left(}
\renewcommand{\r}{\right)}

\begin{document}
\title{Orbits of homogeneous polynomials on Banach spaces}
\author{Rodrigo Cardeccia, Santiago Muro}

\address{DEPARTAMENTO DE MATEM\'ATICA - PAB I,
	FACULTAD DE CS. EXACTAS Y NATURALES, UNIVERSIDAD DE BUENOS AIRES, (1428) BUENOS AIRES, ARGENTINA AND IMAS-CONICET} \email{rcardeccia@dm.uba.ar} 
\address{FACULTAD DE CIENCIAS EXACTAS, INGENIERIA Y AGRIMENSURA, UNIVERSIDAD NACIONAL DE ROSARIO, ARGENTINA AND CIFASIS-CONICET}
\email{muro@cifasis-conicet.gov.ar}

\thanks{Partially supported by ANPCyT PICT 2015-2224, UBACyT 20020130300052BA, PIP 11220130100329CO and CONICET}


\begin{abstract}
We study the dynamics induced by homogeneous polynomials on Banach spaces. It is known that no homogeneous polynomial defined on a Banach space can have a dense orbit. We show, a simple and natural example of a homogeneous polynomial with an orbit that is at the same time $d$-dense (the orbit meets every ball of radius $d$), weakly dense and such that $\Gamma \cdot Orb_P(x)$ is dense for every $\Gamma\subset \mathbb C$ that is either unbounded or that has 0 as an accumulation point.
  Moreover we generalize the construction to arbitrary infinite dimensional separable Fr\'echet spaces. To prove this we study Julia sets of homogeneous polynomials on Banach spaces.
\end{abstract}

\maketitle

\section{Introduction}
Let $X$ be a Banach space. A function $F:X\to X$ is said to be \textit{hypercyclic} if there exists $x\in X$ such that its orbit $Orb_F(x)=\{F^n(x):n\in\zN_0\}$ is dense in $X$. In this case $x$ is called a \textit{hypercyclic} vector.
 The theory of linear hypercyclic operators had a great development in the last decades (see for example the books on the subject \cite{BayMat09,GroPer11}). 

As a natural extension of the linear theory on may study orbits of (non linear) homogeneous operators, or polynomials. The first result in this direction was due to Bernardes \cite{Ber98}, where he proved that no (non linear) homogeneous polynomial on a Banach space can be hypercyclic.
From his result, it can be deduced that associated to each homogeneous polynomial there is a ball (afterwards \textit{ the limit ball}) which is invariant under the action of the polynomial. Moreover, any orbit which enters the limit ball converges to 0. 
 Maybe this result was one of the reasons why the theory of the dynamics of non linear, and in particular, homogeneous polynomials has had a much smaller development than the linear counterpart. However, the behavior of the orbits induced by a homogeneous polynomial can be highly nontrivial and it is far from being understood. For example, in \cite{Ber98} Bernardes showed that the orbits may oscillate between infinity and the boundary of the limit ball. He also proved that every infinite dimensional and separable Banach space supports supercyclic homogeneous polynomials. More recently Peris, Kim and Song \cite{KimPer12,KimPer122} proved that every separable Banach space of dimension greater than one supports \textit{numerically hypercyclic} homogeneous polynomials. This means that there are vectors $x\in S_X, x^*\in S_{X^*}$ for which its numerical orbit, $Norb_P(x,x^*):=\{ x^*(P^n(x)):\,n\in\mathbb N_0\}$ is dense in $\C$.

Since any orbit that meets the limit ball converges to 0, no orbit induced by a homogeneous polynomial is dense in $X$. But, it is  natural to ask how big and complicated can the orbits that never meet the limit ball be. For example, we can ask whether there exist homogeneous polynomials with orbits which imply some weaker types of hypercyclicity.
In particular we can formulate the following questions:
\begin{enumerate}
	\item The Bourdon-Feldman Theorem \cite{BouFel03} states that somewhere dense orbits of linear operators are actually dense.  Does there exist homogeneous polynomials in Banach spaces with somewhere dense orbits?
	\item Feldman \cite{Fel02} showed that if a linear operator is $d$-hypercyclic for some $d>0$ ( i.e. an operator with an orbit that meets every ball of radius $d$), then it is hypercyclic. Does there exist $d$-hypercyclic homogeneous polynomials in Banach spaces?
	\item Are there weakly hypercyclic homogeneous polynomials (i.e. polynomials with an orbit which is dense with respect to the weak topology)?
	\item Recently, Charpentier, Ernst and Menet \cite{ChaErnMen16} characterized the subsets $\Gamma\subset\mathbb C$ for which every  $\Gamma$-supercyclic  linear operator (i.e. operators $T$ such that $\Gamma\cdot Orb_T(x)$ is dense for some $x\in X$) is a hypercyclic linear operator. For which $\Gamma\subset\mathbb C$ are there $\Gamma$-supercyclic homogeneous polynomials?
\end{enumerate}

To study these questions we propose a notion of Julia set associated to homogeneous polynomials in Banach spaces. From its basic properties we deduce that  every orbit of a homogeneous polynomial is nowhere dense. Thus, the Bourdon-Feldman Theorem still holds for homogeneous polynomials. Then, a careful study of the Julia set of a very simple homogeneous polynomial on $\ell_p$, which is the product of the backward shift operator and a linear functional, allows us to prove that it is, at the same time, $d$-hypercyclic, weakly hypercyclic and $\Gamma$-supercyclic  for every $\Gamma\sub \C$ that is either  unbounded or not bounded away from zero. Moreover all this properties are achieved by the same orbit. Finally, we also show that such a homogeneous polynomial exists on every separable infinite dimensional Banach space.




It should be mentioned that if the space is a non-normable Fr\'echet space, then it can support hypercyclic homogeneous polynomials. The first to notice it was Peris \cite{Per01}, who exhibited a chaotic homogeneous polynomial on $\C^\zN$. Later,  other examples were presented, on some K\"othe Echelon spaces (including $H(\zD)$) \cite{MarPer10}, on some spaces of differentiable functions on the real line \cite{AroMir08} and more recently on $\H$, the space of entire functions on the complex plane \cite{cardeccia2017hypercyclic}.
There are also positive results for hypercyclicity of non-homogeneous polynomials. In \cite{MarPer09} the authors proved that every complex separable infinite dimensional Banach space supports a non-linear hypercyclic polynomial. More recently Bernardes and Peris showed in \cite{BerPer13} the existence of 
frequently hypercyclic, chaotic and distributionally chaotic 
 (non-homogeneous) polynomials for a very wide class of infinite dimensional separable Banach spaces. 
There where also some attempts to extend the concept of hypercyclicity to multilinear operators, see \cite{GroKim13,BesCon14}.

\newpage
\section{Julia sets for homogeneous polynomials on Banach spaces}
\subsection{Preliminaries}
Given a topological space $X$, a function $F:X\to X$ is said to be \textit{transitive} if for every pair of non empty sets $U,V$, there exists $n\in\zN$ such that $F^n(U)\cap V\ne\emptyset$. If $X$ is a complete separable metric space without isolated points, then transitivity is equivalent to hypercyclicity. This means that there exists $x\in X$ whose orbit, $Orb_F(x)=\{F^n(x):n\in\zN\}$, is dense in the space. If $F$ is hypercyclic and the periodic vectors are dense in the space, then $F$ is said to be \textit{chaotic}.

Given an $F$-space $X$,  $P:X\to X$ is an $m$-homogeneous polynomial if there exist an $m$-linear operator $A:\overbrace{X\times\ldots \times X}^m\to X$ such that $P(x)=A(x,\ldots,x)$. 
We will deal with continuous polynomials. If $X$ is a normed space then we can define a norm in the space of $m$-homogeneous polynomials $\PP(^m X;X)$,
$$\|P\|:=\sup_{x\in B_X} \|P(x)\|.$$

Bernardes showed that if $X$ is a normed space, then no (non linear) homogeneous polynomial is hypercyclic \cite{Ber98} (see also \cite[Lemma 1.2]{Ued94} for a related finite dimensional statement). For the sake of completeness and because we need to define the notion of limit ball, we give a proof of (slightly modified version of) the mentioned result.
\begin{proposition}
	\label{Ber}
  Let $X$ be a normed space, let $P$ be an $m$-homogeneous polynomial with $m\geq 2$ and let $r_P= \|P\|^\frac{1}{1-m}$. Then $P(r_P\overline B_X)\subseteq r_P\overline B_X$. Moreover for all $y\in r_PB_X$, $\lim_{n\rightarrow \infty} P^n(y)=0$. 
\end{proposition}

\begin{proof}
 Let $x\in r_P\overline B_X$,then
\begin{align*}
 \|P(x)\|&\leq \|P\|\|x\|^m< \|P\|\|P\|^{\frac{m}{1-m}}\\
 &=\|P\|^{1+\frac{m}{1-m}}=r_P.
\end{align*}
Let $x$ with $\|x\|<r_P$, then
$$\|P^n(x)\|< \left(\frac{\|x\|}{r_P}\right)^{m^n} r_P\rightarrow 0.$$
\end{proof}










\begin{corollary}
 No homogeneous polynomial of degree $\geq 2$ defined on a normed space  can be hypercyclic.
 \end{corollary}

 \begin{definition}  
  Let $P$ be an $m$-homogeneous polynomial on a normed space $X$. We will say that $r_P=\|P\|^{\frac{1}{1-m}}$ is the limit radius of $P$ and that $r_PB_X$ is the limit ball of $P$.
 \end{definition}

 \subsection{Julia sets for homogeneous polynomials}
 One distinctive feature of the dynamics induced by a holomorphic function on the complex plane is that we may partition its domain into two completely invariant sets: the Julia set, which is closed, non-empty, perfect, with chaotic behavior (and thus, where the interesting dynamics occur) and the Fatou set, which is open and
 has regular behavior. Along the years many generalizations to several variables appeared and there is not a uniform consensus on what the Julia set means in this context. 
 

Note that we can naturally 
partition the dynamical system $(P,X)$ induced by a homogeneous polynomial in two different systems. On the one hand we have $A_P:= \bigcup_{n\ge 0} P^{-n}(r_PB_X)=\{x\in X: P^n(x)\rightarrow 0\}$. 
 This set is clearly an open set, $P$-invariant and is by definition the basin of attraction of zero. On the opposite side we have its complement, namely $A_P^c=\{x:\|P^n(x)\|\geq r_P \ \forall\, n\ge 0\}$. Clearly $A_P^c$ is closed and it is also $P$-invariant. For the aims of this article, we propose the following. 

 \begin{definition}\rm
The \textit{Julia set} associated to the homogeneous polynomial $P$ on a normed space is the set $J_P:= \partial A_P$.
 \end{definition}
As the following easy result shows, any orbit with interesting dynamics must lie in the Julia set.
 \begin{proposition}\label{estoy en J_P}
 Let $P$ be an $m$-homogeneous polynomial on a normed space $X$. If $x\notin J_P$ then either $\lim P^n(x)= \infty$ or $\lim P^n(x)= 0$.
 \end{proposition}
 \begin{proof}
 Let $x\notin J_P$ and suppose  $ P^n(x) \nrightarrow 0$. Then $x\notin A_P$ and consequently 
 there is some $0<t<1$ with $tx\notin A_P$. This implies by Proposition \ref{Ber} that, for every $n$, $\|P^n(tx)\|\geq r_P$ and therefore $\|P^n(x)\|=\frac{1}{t^{m^n}}\|P^n(tx)\|\geq \frac{1}{t^{m^n}} r_P \to\infty$. 
 \end{proof}
 Notice that, in particular, periodic vectors and orbits having nonzero accumulation points must belong to $J_P$.

 In the linear operator setting, an irregular vector is a vector that  satisfies $\liminf_n \|T^n(x)\|=0$ while $\limsup_n \|T^n(x)\|=\infty$. This conditions are incompatible for a non-linear homogeneous polynomial on a Banach space. In view of Proposition \ref{Ber}, it is natural to give the following definition of irregular vectors for homogeneous polynomials. Note that any irregular vector must belong to $J_P$.
 \begin{definition}
A vector is \textit{irregular} for a homogeneous polynomial $P$ if $\liminf_n \|P^n(x)\|=r_P$ and $\limsup_n \|P^n(x)\|=\infty$.
 \end{definition}

 We may also consider the basin of attraction of infinity. We define
 $R_P$ as the maximal open subset such that every orbit tends in norm to $\infty$, that is, $R_P=\{x\in X:\text{ there exist } r>0 \text{ such that } \|P^n(y)\|\to \infty \text{ for every } y\in B_r(x)\}$.
 By definition $R_P$ is open and contained in $\overline{A_P}^c$.
 \begin{proposition}
 Let $P$ be an $m$-homogeneous polynomial defined over a normed space $X$. Then, $X$ is the disjoint union of $A_P,\, J_P$ and $ R_P$.
 \end{proposition}
\begin{proof}
 It suffices to prove that $(A_P^c)^\circ=R_P.$
 Since $R_P$ and is open $R_P\cap A_P=\emptyset$ we have that $R_P\sub (A_P^c)^\circ$. 
 Reciprocally, if $x\in (A_P^c)^\circ$, there exists an $\epsilon>0$ with $B_\epsilon (x)\cap \overline {A_P}=\emptyset$. Let $y\in B_\epsilon(x)$, by the last proposition either $P^n(y)\to 0$ or $\|P^n(y)\|\to \infty$. Since $y\notin A_P$, $P^n(y)\to \infty$.
 Therefore $x\in R_P$.
\end{proof}
Recall that a function $F:Y\to Y$ is said to be \textit{quasiconjugate} to a function $G:X\to X$ if there exist a continuous factor $\phi:X\to Y$, with dense range, such that the following diagram commutes
\[
\xymatrix{
	X \ar[d]^\phi \ar[r]^G & X \ar[d]^\phi \\
	Y \ar[r]^F & Y
}
\]
Most of the dynamical properties are preserved under quasiconjugacy.

The dynamical system induced by a linear operator $T$ and $\lambda T$ may be completely different. Indeed, if $\|\lambda T\|<1$ then every orbit tends to zero while $T$ can support dense orbits. This is not the case if $P$ is a homogeneous polynomial. The same phenomenon occurs for bilinear operators (see \cite{GroKim13}).
\begin{proposition}\label{quasiI}
	Let $P$ be an $m$-homogeneous polynomial, $m>1$. Then for every $\lambda\neq 0$, $\lambda P$ is quasiconjugate to $P$ under a linear isomorphism.
\end{proposition}

\begin{proof}
	Let $\lambda^\f{1}{m-1}$ be any $\f{1}{m-1}$-root for $\lambda$. The following factor works, $\phi(x)=\f{1}{\lambda^\f{1}{m-1}} x$. Indeed,
	\begin{align*}
	\lambda P\l\f{x}{\lambda^\f{1}{m-1}}\r&=
\f{\lambda}{\lambda^\f{m}{m-1}} P(x)=\f{1}{\lambda^\f{1}{m-1}}P(x).
	\end{align*}
\end{proof}
The Julia set is preserved under quasiconjugacy provided that the factor is a linear isomorphism.
\begin{lemma}\label{conjugar J_P}
 Let $Q$ and $P$ be homogeneous polynomials such that $Q$ is quasiconjugate to $P$ under a linear isomorphism. 
Then $\phi (A_P)=A_Q$, $\phi(J_P)=J_Q$, $\phi(R_P)=R_Q$.
\end{lemma}
\begin{proof}
 Since $\phi$ is an isomorphism, it is enough to prove that $\phi(A_P)\sub A_Q$ and $\phi(R_P)\sub R_Q$.
 Let $x\in A_P$, then $\|Q^n(\phi(x))\|=\|\phi(P^n(x))\|\leq \|\phi\|\|P^n(x)\|\rightarrow 0$. So that $\phi(x)\in A_Q.$
 Consider now $x_0\in R_P$, and let $\epsilon'>0$ such that $\|P^n(x)\|\rightarrow \infty$ for all $x\in B_{\epsilon'}(x_0)$. Let
  $\epsilon>0$ such that $\phi^{-1}(B_{\epsilon}(\phi(x_0))\sub B_{\epsilon'}(x_0)$. Thus for $y\in B_{\epsilon}(\phi(x_0))$, 
 $\|Q^n(y)\|=\|Q^n(\phi(\phi^{-1}(y)))\|=\|\phi(P^n(\phi^{-1}(y)))\|$, and thus $\|Q^n(y)\|\rightarrow \infty$ because $\phi$ is an isomorphism and $\phi^{-1}(y)\in B_{\epsilon'}(x_0)$.
 \end{proof}
\begin{corollary}
	Let $P$ be a non linear $m$-homogeneous polynomial and $\lambda\neq 0$. Then   $J_{\lambda P}=\f{1}{\lambda^{\f{1}{m-1}}} J_P$.
\end{corollary}
\begin{lemma}
If $x\in A_P$ then $tx\in A_P$ for all $|t|\le 1$, if $x\in R_P$ then $tx\in R_P$ for all $|t|\ge 1$.
\end{lemma}
\begin{proof}
 Let $x\in A_P$ and $|t|\le 1$. Then $\|P^n(tx)\|\leq \|P^n(x)\|\rightarrow 0$. Consequently we have that $tx\in A_P$.
 Suppose now that $x\in R_P$ and $|t|\ge 1$. There exists $\epsilon>0$ such that $\|P^n(y)\|\rightarrow 0$ for all $y\in B_\epsilon(x)$.
 Let $y\in B_{\epsilon t}(tx)$, therefore
 \begin{equation*}
  \left\|\frac{y}{t}-x\right\|\leq \epsilon \text{ and }\|P^n(y)\|\geq\|P^n\l\frac{y}{t}\r\|\rightarrow \infty.
  \end{equation*}
 \end{proof}
 The following result shows that the Julia set of a homogeneous polynomial share some of the properties satisfied by the Julia set of  a holomorphic function on the complex plane.
\begin{proposition}[Properties of $J_P$]
Let $P$ be a non linear homogeneous polynomial.
The Julia set $J_P$ satisfies the following properties:
\begin{itemize}
\item[i)] it is a closed set with empty interior;
\item[ii)] it is $P$-invariant;
\item[iii)] it is perfect;
\item[iv)]  $J_{P^n}=J_P$.
\end{itemize}  
\end{proposition}
\begin{proof}
i) This is clear since $J_P$ is the boundary of an open set.

ii) First note that $y\in A_P$ if and only if $P(y)\in A_P$. Thus, if $x\in J_P$ then $P(x)$ is not in $A_P$. On the other hand,
there exists  $(a_n)_n\sub A_P$ such that $a_n\to x$. Since  $P(a_n)$ belongs also to $A_P$ and  $P(a_n)\to P(x)$, we have that $P(x)\in J_P$.

iii) Just note that if $x\in J_P$ then  $\lambda x\in J_P$ for every $|\lambda|=1$. 

iv) It suffices to show that $A_{P^n}=A_P$. Clearly if $P^k(x)\to 0$ then $(P^n)^k(x)\to 0$ and therefore $A_{P}\sub A_{P^n}$. The converse follows thanks to the existence of the limit ball (Proposition \ref{Ber}). If $P^{nk}(x)\to 0$ then  $Orb_{P^n}(x)$ meets eventually the limit ball of $P$ and therefore $P^n(x)$ must tend to zero. 
\end{proof}
 In contrast to the one dimensional case, the Julia set $J_P$ may be empty. Indeed,  if $P\in \PP(^2 \ell_2;\ell_2)$ is defined as $P(x)=(x_2^2,x_3^2,x_4^2,\ldots)$,  then $P^n(x)\to 0$ for every $x\in \ell_2$. Thus $A_P=\ell_2$ and hence $J_P=R_P=\emptyset$.

Recall that a set $A$ is \textit{completely invariant} under $P$ if $P(A)\sub A$ and $P^{-1}(A)\sub A$. The Julia set of a homogeneous polynomials need not to be completely invariant, as we will show in Example \ref{no completamente inv}.
On the other hand, under certain conditions on $P$, $J_P$ results completely invariant.

\begin{proposition}\label{completely invariant}
	Let $P\in \PP(^d X;X)$. If $P$ is open or if $R_P=\emptyset$ then $J_P$ is completely invariant.
\end{proposition}
\begin{proof}
Let $x\in X$ such that $P(x)\in J_P$. Suppose that $x\notin J_P$, so that $x\in A_P$ or $x\in R_P$. If $x\in A_P$ we have that $P^n(x)\to 0$ and hence $P(x)\in A_P$.
 If $R_P=\emptyset$ this implies that $x\in J_P$. If $R_P\neq \emptyset$ and $x\in R_P$, there is some $\epsilon>0$ so that $\|P^n(y)\|\to \infty$ for every $y\in B_\epsilon (x)$. Since $P$ is open, $P(B_\epsilon (x))$ is an open neighborhood of $P(x)$ and $P^n(z)\to \infty$ for every $z\in P(B_\epsilon (x))$ and hence $P(x)\in R_P$.
\end{proof}
Since $J_P$ has empty interior, we can extract an easy but important corollary, which may be seen as an extension to the nonlinear case of the Bourdon-Feldman Theorem \cite{BouFel03}.
\begin{corollary}
Let $X$ be a normed space, and $P\in \mathcal P(^mX;X),$ $m\geq 2$.
Then the following subsets are always nowhere dense: any orbit induced by $P$, the set of irregular vectors, the set of periodic vectors.
\end{corollary}
\begin{proof}
 The periodic and irregular vectors are contained in the Julia set, which is nowhere dense.
 In order to be a somewhere dense orbit, the orbit can not tend to infinity or to zero. Therefore it must be contained in the Julia set, which has empty interior.
\end{proof}
It is natural to ask now how big the Julia sets (or the closure of an orbit) can be. We already know that it cannot contain interior points. In the next subsection we will see some examples of Julia sets, but let us first show that no multiple of the limit sphere can be contained in the Julia set. We will need the following.
 \begin{proposition}\label{seg}
Let $P$ be an $m$-homogeneous polynomial, and let $z\in J_P$ such that $tz\in J_P$ with  $0<|t|<1$. Then $P^n(z)\to\infty$. Consequently if 
 $x\in X$, $y\in Orb_P(x)$ and $0<|t|<1$, then 
$ty$ is not an accumulation point of ${Orb_P (x)}.$
\end{proposition}
\begin{proof}
If	$tz\in J_P$ then $\|P^n(tz)\|\geq r_P$ for every $n\in\zN$. Thus,
	$$\|P^n(z)\|=\frac{1}{|t|^{m^n}}\|P^n(tz)\|\geq \frac{r_P}{|t|^{m^n}}\to \infty.$$ 
For the last assertion, suppose otherwise. It follows that $y\in Orb_P(x)\subset J_P$. Since $J_P$ is closed, $ty\in J_P$. Thus, $P^n(y)\nrightarrow \infty$. This is a contradiction since $ty$ is an accumulation point of $Orb_P(y)$
\end{proof}
We will see in the next section that it is possible that the orbit of a point $y$ accumulates at a multiple $ty$, for $|t|\ge1$ (we will give an example that accumulates at every such a multiple). 
\begin{corollary}
	Let $P$ be an $m$-homogeneous polynomial and let $s\ne r_P=\|P\|^{-\frac1{m-1}}$. Then $sS_X\nsubseteq J_P$. In particular, $sS_X$ is not contained in the closure of any orbit of $P$.
\end{corollary}
\begin{proof}
	For $s<r_P$ the result is a consequence of Proposition \ref{Ber}.
	Suppose that $sS_X\subset J_P$, for some $s>r_P$. By Proposition \ref{seg}, if $\|z\|>s$ then $P^n(z)\to\infty$. Thus $\{\|z\|>s\}\subset R_P$, 
	and since $J_P$ is invariant, $P(s\overline B_X)\subset s\overline B_X$.
	
	Since $s^{m-1}\|P\|= \frac{s^{m-1}}{r_P^{m-1}}>1$, there is some $x_0\in S_X$ with $\|P\|< s^{m-1}\|P\|\|P(x_0)\|$. Then, 
	$$
	\|P(sx_0)\|>\frac{s^m}{s^{m-1}\|P\|}\|P\|=s,
	$$
	which is a contradiction. 
\end{proof}
%

In the last proposition it was important that $s\neq r_P$, and it is not true for $s=r_P$. Indeed, we have the following easy example:  let $X=\ell_\infty^n=(\C^n,\|\cdot\|_\infty)$ and let $P:X \rightarrow X$ be the 2-homogeneous polynomial defined as $P(z_1,\dots,z_n)= (z_1^2,\dots,z_n^2)$. Since $\|P\|=1$,  the limit sphere is the unit sphere and it is easy to see that it coincides with $J_P$. If $n=1$ the polynomial is $P(z)= z^2$ defined in $\C$. In this trivial example $J_P=\mathbb T$ and, furthermore $P|_{\mathbb T}$ is the doubling map on the circle, which is mixing and therefore has dense orbits. 
We do not know if there exists a non-trivial example of a Banach space $X$ and a homogeneous polynomial $P$ on $X$ having a dense orbit in its limit sphere.
Note that such a polynomial must satisfy that  $r_PS_X=J_P$ is completely invariant (thus for every $x\in X$, $\|P(x)\|=\|P\|\|x\|^m$) and that $P|_{r_PS_X}$ is transitive.
%
%

\subsection{Some examples of Julia sets}

In the following example, the Julia set contains an affine hyperspace.



\begin{example}\label{affin hyperspace}
 Let $X$ be a separable infinite dimensional Banach space. Take  $\varphi\in X^*$ and $x_0$ such that $\varphi(x_0)=1$. Since $Ker(\varphi)$ is an infinite dimensional Banach space, there exists  a hypercyclic operator $T:Ker(\varphi)\to Ker(\varphi)$. We define $P\in \mathcal P(^mX;X)$ as
 $$P(x)=\varphi(x)^m x_0 + \varphi(x)^{m-1}T(x-\varphi(x)x_0).$$
 Then $J_P=\zT x_0\oplus Ker(\varphi)$ and $P|_{J_P}$ is transitive.

 \end{example}

\begin{proof}
 We start by studying $A_P$ and $R_P$. Let $x=\lambda x_0 + y$ with $|\lambda|<1$ and $y\in Ker(\varphi)$. Since 
 \begin{equation}\label {P affin}
 P^n(x)=\lambda^{m^n}x_0+ \lambda^{m^n-1}T^n(y)
 \end{equation}
 we have that $\|P^n(x)\|\leq |\lambda|^{m^n}\|x_0\|+|\lambda^{m^n-1}|\|T(y)\|^n\to 0$. Thus $A_P\sub \{x:|\varphi(x)|<1\}$.
 
 Consider in $X$ the following equivalent norm $\|x\|_\infty:= \max\{|\varphi(x)|\|x_0\|,\|x-\varphi(x)x_0\|\}$.
By Lemma \ref{conjugar J_P}, $R_P$, $A_P$ and $J_P$ are invariant under equivalent norms.  
The set $\{x:|\varphi(x)|>1\}$ is open and, by \eqref{P affin}, if $|\varphi(x)|>1$, then $\|P^n(x)\|_\infty\to \infty$. Therefore $\{x:|\varphi(x)|>1\}\sub R_P$, and  $\{x:\varphi(x)\in \zT\}=\partial A_P=J_P$.

%
%
 
Finally, let us see that  $P|_{\zT x_0 \oplus Ker(\varphi)}$ is transitive.  
 Recall that for every open set $W\sub\mathbb \zT$, there exists $n_0$ such that  $\zT \sub W^{m^n}$ for every $n\geq n_0$. 
 Let $U$ and $V$ be open sets in $\zT x_0+ Ker(\varphi)$. By considering the norm $\|\cdot\|_\infty$ in $X$, we can suppose that $U=(U_1,U_2)$, $V=(V_1,V_2)$ with $U_1,V_1\sub \zT x_0$, $U_2,V_2\sub Ker(\varphi)$, all of them being open sets.
 Let $n_0\in \zN$ be with $\zT\sub U_1^{m^{n}-1}$ for every $n\geq n_0$. Since $T$ is transitive, there is $n_1>n_0\in \zN$ with $V_2\cap T^{n_1}(U_2)\neq \emptyset$. 
 
 Note that, by \eqref{P affin}, 
 $$
 P^{n_1}(U_1\times U_2)=\{\lambda^{m^{n_1}}x_0+ \lambda^{m^{n_1}-1}T^{n_1}(y): \,\lambda x_0\in U_1,\, y\in U_2\}= \zT x_0 \times\zT T^{n_1}(U_2).
 $$
 Therefore,
 $P^{n_1}(U_1\times U_2)\cap V_1\times V_2\neq \emptyset.$
 \end{proof}
Bernardes proved in \cite[Theorem 2]{Ber98} that if $X$ is an infinite dimensional separable Banach space then there exist a supercyclic homogeneous polynomial acting on $X$. 
Notice that the fact that $P|_{\zT x_0\oplus Ker(\varphi)}$ is transitive implies that  in particular $P$ is supercyclic (and also $\mathbb R_{+}$-supercyclic). Moreover, such a polynomial may be also constructed in any separable infinite dimensional  Fr\'echet space. Thus, we have given a simple proof and an extension of \cite[Theorem 2]{Ber98}.

\begin{proposition}\label{existen superciclos}
  For any infinite dimensional separable Fr\'echet space $X$, and every natural $m\ge 2$ there exists a homogeneous polynomial $P\in\mathcal P(^mX; X)$ which is $\mathbb R_{+}$-supercyclic. 	
\end{proposition}

\begin{remark}\rm
The study of the dynamics of homogeneous polynomials is closely related to the study of the iterations of holomorphic and meromorphic mappings on projective spaces. Each meromorphic mapping on $\mathbb P^k$ has a lifting to $\mathbb C^{k+1}$ which is a homogeneous polynomial. 
Moreover,  a homogeneous polynomial on $\mathbb C^{k+1}$ is supercyclic if and only if its associated mapping on the projective space has an orbit which is dense in $\mathbb P^k$. The first example of such a mapping was constructed by L\`attes on $\mathbb P^1$ and this can be also done in $\mathbb P^k$, see for example \cite{dinh2010dynamics}. Thus, we also know the existence of supercyclic homogeneous polynomials on finite dimensional spaces. 
\end{remark}

 Bernardes in \cite[Proposition 5]{Ber98} gave an example of homogeneous polynomial $P$ with an irregular vector $x_0$. In his example, $Orb_P(x_0)$ does not have accumulation points and moreover, the sequence $(\|P^n(x_0)\|)_n$ is an unbounded sequence of positive numbers whose only accumulation point  is $r_P$. Next, we present an example of a norm one homogeneous polynomial (thus $r_P=1$), with an irregular vector $x$ such that $\overline{Orb_P(x)}$ is an affine hyperplane touching the sphere of radius $r_P=1$, and thus $\overline{\{\|P^n(x)\|:n\in\zN\}}=[r_P,\infty)$.

 \begin{example}
 Let $P$  be the $m$-homogeneous polynomial defined in the Example \ref{affin hyperspace}. We set $m=2$,  $X=\ell_1$, $x_0=e_1$, $\varphi=e_1'$, and $T:Ker(\varphi)\equiv \ell_1(\mathbb N_{\ge2})\rightarrow \ell_1(\mathbb N_{\ge2})$ to be the hypercyclic operator $(1+\epsilon) B$, where $B$ is the backward shift on $\ell_1(\mathbb N_{\ge2})$ and $0<\epsilon<1$.
 
Then $\|P\|=1$ and $J_P=\zT e_1 + Ker(e_1')$.

Since $P|_{J_P}$ is transitive, there exists an  orbit whose closure contains an affine hyperspace that meets the limit ball. In particular there is an irregular vector $x$ such that $\overline{\{\|P^n(x)\|:n\in\zN\}}$ is $[1,\infty)=[r_P,\infty)$. 
 \end{example}
\begin{proof}
By the previous example, it suffices to prove that $\|P\|=1$.

 Let $x$ be in $S_{\ell_1}$ and suppose that $x=\lambda e_1 + \gamma$, where $\gamma\in Ker(e_1')$ and $|\lambda|+\|\gamma\|_1=1$. 
 Therefore 
 $$\|P(x)\|_1=|\lambda^2|+|\lambda|\|(1+\epsilon)B((\gamma_n))\|_1\le |\lambda|^2+(1+\epsilon)|\lambda| \|(( \gamma_n)_n)\|_1,$$
 and $\lambda+\|(\gamma_n)_n\|_1=1.$
  


Thus, $\|P\|$ is less than or equal to the maximum of the function 
 $f(s,t)=s^2+(1+\epsilon)st$ restricted to
 $s+t=1, s\geq 0$ and $t\geq 0$. Since $0<\epsilon<1$, this maximum is exactly 1. 
 
 On the other hand, $P(e_1)=e_1$, and therefore $\|P\|=1$.
%
%
%
%
\end{proof}
It is also possible to construct an  analogous $m$-homogeneous polynomial on $X=\ell_p$, $1<p<\infty$, provided $m>(1+\epsilon)^p$.

%
%
%
%
%
%
%
%
In \cite{Per01} it was shown that the polynomial defined as the backward shift to the power of $m$, acting on the (non-normable) Fr\'echet space  $\mathbb C^{\mathbb N}$ is chaotic (see also \cite{MarPer10}). On $\ell_p$ ($p<\infty$) or $c_0$ the dynamic of this polynomial is trivial since every orbit converges to 0, i.e. $A_P$ is the whole space and the Julia set is empty. The next example analyzes this polynomial in the space of convergent sequences.
 \begin{example}\label{J_P de c}
  Take $X=c$ the space of convergent sequences and consider the $m$-homogeneous polynomial defined as the backward shift to the power of $m$, that is, $P(x)_j=x_{j+1}^m$. Then $A_P=\{(a_j)_j:|\lim_j a_j|<1\}$, $J_P=\{(a_j)_j:|\lim_j a_j|=1\}$, $R_P=\{(a_j)_j:|\lim_j a_j|>1\}$ and $P|_{J_P}$ is transitive.
 \end{example}
\begin{proof}
Observe first that $\|P\|=1=r_P$.

Let $(a_n)_n\in c$ with $L=|\lim a_n|<1$. Let $n_0\in\zN$ such that $|a_n|<\delta <1$ for every $n\geq n_0$. Therefore $\|P^n((a_n)_n)\|_\infty\leq \delta^{m^n}\to 0$. This implies that $\{(a_j)_j:|\lim_j a_j|<1\}\sub A_P$. Similarly $R_P\sub \{(a_j)_j:|\lim a_j|>1\}$.
Finally since $\{(a_j)_j: |\lim a_j|=1\}\sub\partial A_P$ it follows that $J_P=\{(a_j)_j:|\lim a_j|=1\}$.

Let us now show that  $P|_{J_P}$ is transitive. Notice that $A:=\{a+e^{i\theta}\mathbb 1:a\in c_{00},\theta \in [0,2\pi]\}$ is dense in $J_P$, where $\mathbb 1$ is the vector with $[\mathbb 1]_i=1$ for all $i$.

Let $U$, $V$ be open sets and $v+e^{i\theta_2}\mathbb 1\in V\cap A$. Let $\epsilon >0$ with $B_\epsilon (v+e^{i\theta_2})\sub V$. Since the map $\Phi:\zT\rightarrow \zT$, $\Phi(e^{i\theta})=e^{mi\theta}$ induces a mixing dynamical system, there is $u+e^{i\theta_1}\mathbb 1\in U\cap A$ such that the $\Phi$-orbit of $e^{i\theta_1}$ is dense in $\zT$. Suppose that $u=\sum_{k=1}^Nu_ke_k$ and $v=\sum_{k=1}^Nv_ke_k$.

In the following, denote by $a^{\frac{1}{n}}$, for each  $0\ne a\in\mathbb C$ and for each $n$, a $\frac{1}{n}$-root of $a$ so that $a^{\frac{1}{n}}\to 1$, as $n\to\infty$.  
Thus, there exists $n_0>N$ such that for any $n\ge n_0$, the vector $x_n$ belongs to $U$, where
$$
x_n:=\sum_{k=1}^N(u_k+e^{i\theta_1})e_k+\sum_{k=N+1}^{n}e^{i\theta_1}e_k+\sum_{k=n+1}^{N+n}(v_{k-n}e^{-i\theta_2}+1)^{\frac{1}{m^n}}e^{i\theta_1}e_k+\sum_{k=N+n+1}^{\infty}e^{i\theta_1}e_k.
$$
Therefore 
\begin{align*}
 \|P^{n}(x_{n})-v-e^{i\theta_2}\mathbb 1 \|_\infty&=\|
 \sum_{k=1}^{N}(v_{k}e^{-i\theta_2}+1)e^{i\theta_1m^n}e_k+\sum_{k=N+1}^{\infty}e^{i\theta_1m^n}e_k
-v-e^{i\theta_2}\mathbb 1\|_{\infty}\\
 &
 =\|\sum_{k=1}^{N}v_{k}(e^{-i\theta_2}\Phi^n(e^{i\theta_1})-1)e_k+\sum_{k=1}^{\infty}(\Phi^n(e^{i\theta_1})-e^{i\theta_2})e_k\|_{\infty}\\
 &\le (\|v\|_\infty+1)|\Phi^n(e^{i\theta_1})-e^{i\theta_2}|.
\end{align*}
Thus, by the election of $\theta_1$, we can take $n\ge n_0$ such that $|\Phi^n(e^{i\theta_1})-e^{i\theta_2}|$ is arbitrarily small, so that $P^n(x_n)$ belongs to $V$.
%
%
\end{proof}

 \section{$d$-hypercyclic, weakly hypercyclic and $\Gamma$-supercyclic homogeneous polynomials}
  No homogeneous polynomial on a Banach space can be hypercyclic, so it is natural to ask for the existence of weaker notions of hypercyclicity.
  In this section we will provide a simple and natural homogeneous polynomial that is at the same time $d$-hypercyclic, weakly hypercyclic and $\Gamma$-supercyclic for every unbounded or not bounded away from zero $\Gamma\sub \C$. Let us recall some basic definitions.

 Let $X$ be a Fr\'echet space and $(\rho_n)_n$ a fundamental system of seminorms. We will say that a set $A\sub X$ is \textit{$d$-dense} for a sequence of positive numbers $d=(d_n)_n$ provided that  for every $x\in X$ and every $n$, there exists $a\in A$ for which $\rho_n(x-a)<d_n$.
  We will say that a function $P:X\to X$ is \textit{$d$-hypercyclic}  if there exist $x$ such that $Orb_P(x)$ is $d$-dense in $X$. In this case, we will say that $x$ is a \textit{$d$-hypercyclic} vector for $P$.
 
This phenomenon was studied by N. Feldman in \cite{Fel02} in the context of linear operators on Banach spaces. There he proved that an operator on a Banach space  is hypercyclic if and only if it is $d$-hypercyclic. With a similar proof it can be proven that, if $X$ is a Fr\'echet space, then both notions are still equivalent for linear operators. 

A function $P:X\to X$ is \textit{weakly hypercyclic} if there exists $x\in X$ such that $Orb_P(x)$ is dense with  respect to the weak topology of the space. Since the norm topology is stronger than the weak topology, it is clear that a hypercyclic function is automatically weakly hypercyclic while the converse is not necessarily true. In \cite{ChaSan04} the authors found the first weakly hypercyclic linear operator that is not hypercyclic.  

We will also consider the  weaker notion for hypercyclicity  called $\Gamma$-supercyclicity. If $\Gamma\sub \C$, then we will say that a function $P:X\to X$ is \textit{$\Gamma$-supercyclic} if there exists $x\in X$ (called $\Gamma$-\textit{supercyclic} vector) such that $\Gamma\cdot Orb_P(x)$ is dense in the space.  In the special case that $\Gamma=\C$ or $\Gamma$ is a singleton we recover the notions of supercyclicity and hypercyclicity respectively.
An important result concerning $\Gamma$-supercyclicity, for $\Gamma=\zT$, was due to Le\'on Saavedra and M\"uller \cite{LeoMul04}, where they proved that a linear operator is $\zT$-supercyclic if and only if it is hypercyclic. The case when $\Gamma=\mathbb R_+$ is also called positive supercyclicity (see \cite{saavedra2004positive}).  The concept of $\Gamma$-supercyclicity was recently introduced by Charpentier, Ernst and Menet in \cite{ChaErnMen16}. They proved that the subsets $\Gamma\sub\C$ such that $\Gamma$-supercyclicity is equivalent to hypercyclicity are exactly the subsets $\Gamma$ which are bounded and such that $\Gamma\setminus\{0\}$ is bounded away from zero.

 Every infinite dimensional separable Banach space supports   supercyclic homogeneous polynomial (see \cite{Ber98} or Proposition \ref{existen superciclos}), but no homogeneous polynomial can be hypercyclic. This means that there are $\C$-supercyclic homogeneous polynomials while no polynomial is $\{*\}$-supercyclic. So it is natural to ask for which subsets $\Gamma\sub \mathbb C$  it is possible to have $\Gamma$-supercyclic homogeneous polynomials on Banach spaces.

\subsection{Remarks on $d$-hypercyclicity}
Notice that if $x$ is a $d$-hypercyclic vector, then $x\in J_P$. Indeed, otherwise $P^n(x)\to \infty$ or $P^n(x)\to 0$ contradicting the $d$-density of $Orb_P(x)$. 

If $R_P\ne\emptyset$, then it contains balls of arbitrary large radius, because if $B_\epsilon(y)\sub R_P$ then $B_{|t|\epsilon}(ty)\sub R_P$ for all $|t|>1$. 
Thus we can find balls with arbitrary large radius not meeting $J_P$. Since a $d$-hypercyclic vector must belong to $J_P$, it follows that the polynomial is not $d$-hypercyclic, and then we have the following.
\begin{remark}
 If $R_P\neq \emptyset$ then $P$ is not $d$-hypercyclic.
\end{remark}
%

\begin{proposition}\label{d hiper}
Let $X,Y$ be Fr\'echet spaces with fundamental system of seminorms $(q_n)_n$ and $(p_n)_n$ respectively. Let $F:X\to X$ and $G:Y\to Y$ be continuous functions such that $G$ is a quasiconjugacy of $F$ via a linear operator $\Phi:X\to Y$. If $F$ is $d$-hypercyclic then $G$ is $\tilde d$-hypercyclic for some $(\tilde d_n)_n$. In the Banach case we have that $G$ is $(d\|\Phi\|+\epsilon)$-hypercyclic for every $\epsilon>0$.
  \[
 \xymatrix{
 X \ar[d]^\Phi \ar[r]^F & X \ar[d]^\Phi \\
 Y \ar[r]^G & Y
 }
 \]
 \end{proposition}
 \begin{proof}
 Let $x$ be a $d$-hypercyclic vector for $F$ and $\epsilon>0$. Let $M_n$ and $q_{k_n}$ such that $p_n(\Phi(x))\leq M_n q_{k_n}(x)$ for every $x\in X$.
 
 Let $y\in Y$ and take $x_0\in X$ such that $p_n(\Phi(x_0)-y)<\epsilon$.
 Since $x$ is a $d$-hypercyclic vector, there exists $l>0$ such that $q_{k_n}(F^l(x)-x_0)<d_{k_n}$. Therefore 
 \begin{align*}
 p_n(G^l\Phi (x)-y)&=p_n(\Phi F^l(x)-y)\leq p_n(\Phi F^l(x)-\Phi(x_0))+p_n(\Phi(x_0)-y)\\
 &\leq M_nd_{k_n}+\epsilon:=\tilde d_n. 
 \end{align*}
 \end{proof}
\subsection{Remarks on weak hypercyclicity} 
The orbits lying on $R_P$ tend to infinity fast. This  forces the orbit to be weakly closed and hence orbits on $R_P$ are far from being weakly dense.

\begin{proposition}\label{weakly closed}
Let $x\notin J_P$. Then $Orb_P(x)$ is $w$-closed.
\end{proposition}

For the proof we will need the following result, which may be found in \cite[Proposition 10.1]{BayMat09}. 
\begin{proposition}
If there are some constants $k>0$ and $C>1$ such that $\|x_n\|\geq k C^n$, then $\{x_n:n\in\zN\}$ is weakly closed.
\end{proposition}
\begin{proof}[Proof of Proposition \ref{weakly closed}]
If $x\in A_P$, $P^n(x)\to 0$ and therefore $\{P^n(x):n\in \zN\}$ is weakly closed.
If $x\in R_P,$ there is some $t>1$ such that $\frac{x}{t}\in R_P$. Then $\|P^n(x)\|=t^{m^n}\|P^n(\frac{x}{t})\|\geq t^{m^n}r_p$. By the above proposition $Orb_P(x)$ is weakly closed.
\end{proof}
\begin{corollary}
Let $x$ be a weakly hypercyclic vector for a homogeneous polynomial. Then $x\in J_P$.
\end{corollary}

It is known that homogeneous polynomials are not necessarily weak to weak continuous and therefore $(P,(X,\omega))$ is not a truly dynamical system.
Nevertheless the property of being weakly hypercyclic is preserved under quasiconjugacy provided that the factor $\Phi$ is weak to weak continuous and $\Phi$ has weakly dense range.
\begin{proposition}\label{weakly conj}
Let $X,Y$ be Fr\'echet spaces, $P\in \PP(^mX;X)$, $Q\in \PP(^m Y;Y)$ and $\Phi:X\to Y$ a weak to weak continuous map with weakly dense range such that the following diagram commutes.
 \[
\xymatrix{
	X \ar[d]^\Phi \ar[r]^P & X \ar[d]^\Phi \\
	Y \ar[r]^Q & Y
}
\]
If $P$ is weakly hypercyclic then so it is $Q$.
\end{proposition}
\begin{proof}
	Let $x\in X$ be a weakly hypercyclic vector for $P$. Since $\overline{Orb_P(x)}^\omega=X$ and $\Phi$ is weak to weak continuous with weakly dense range, we have that $Y=\overline{\Phi(Orb_P(x))}^\omega=\overline{Orb_Q (\Phi(x))}^\omega$. 
\end{proof}
\subsection{Remarks and examples on $\Gamma$-supercyclicity}
Recall that the set of $\Gamma\sub \C$ such that a $\Gamma$-supercyclic operator is automatically a hypercyclic operator are exactly the sets $\Gamma$ such that $\Gamma$ is bounded and bounded away from zero. In the same spirit, if $\Gamma$ is bounded and bounded away from zero, then no homogeneous polynomial is $\Gamma$-supercyclic.
\begin{proposition}\label{Gamma vectors}
Let $P$ be a $\Gamma$-supercyclic homogeneous polynomial and $x$ a $\Gamma$-supercyclic vector. If $\Gamma$ is bounded then $x\notin A_P$ and if $\Gamma\setminus\{0\}$ is bounded away from zero then $x \in A_P$. In particular no homogeneous polynomial is $\Gamma$-supercyclic if $\Gamma$ is bounded and $\Gamma\setminus\{0\}$ is bounded away from zero.
\end{proposition}

 \begin{proof}
    Suppose that $x\in A_P$ then $Orb_P(x)$ is bounded. Hence   $\Gamma\cdot Orb_P(x)$ is bounded if $\Gamma$ is bounded. Thus $x$ is not a $\Gamma$-supercyclic vector. 
  
  Suppose that $x\notin A_P$ and that $\Gamma\setminus\{0\}$ is bounded away from zero. By Proposition \ref{Ber}, $\|P^n(x)\|\geq r_P$ for every $n$. Therefore $\Gamma\cdot Orb_P(x)\setminus\{0\}$ is bounded away from zero and  again $x$ cannot be a $\Gamma$-supercyclic vector.
 \end{proof}
A $\Gamma$-supercyclic vector is not necessarily in $J_P$. Indeed, in Examples \ref{unbounded superciclico}, \ref{mainexample}, and \ref{D-superciclico} we show $\Gamma$-supercyclic vectors belonging to $A_P$, $J_P$ and $R_P$, respectively. 

If the factor $\Phi$ is linear, then $\Gamma$-supercyclicity is preserved under quasiconjugacy.
\begin{proposition}\label{conjugar Gamma}
Let $P,Q$ be homogeneous polynomials such that $Q$ is quasiconjugated to $P$ under a linear factor $\Phi$. If $P$ is $\Gamma$-supercyclic, then $Q$ is also $\Gamma$-supercyclic.
\end{proposition}
\begin{proof}
If $\Gamma\cdot Orb_P(x)$ is dense then $\Gamma\cdot Orb_Q(\Phi(x))=\Gamma\cdot \Phi(Orb_P(x))=\Phi(\Gamma\cdot Orb_P(x))$ is also dense because $\Phi$ has dense range.
\end{proof}

As in the linear case we can define the $\Gamma$-transitivity notion. For some examples it will be an useful  tool to prove $\Gamma$-supercyclity. 

\begin{definition}
	Let $X$ be a separable Fr\'echet space. We say that a polynomial $P\in \PP(^m X;X)$ is $\Gamma$-transitive, if for every open sets $U$ and $V$, there exist $n\in \zN$ and $\lambda\in\Gamma$ with
	$$\lambda P^n(U)\cap V\neq\emptyset.$$ 
\end{definition}
It is easy to prove that $\Gamma$-transitivity implies $\Gamma$-supercyclity. Moreover, in this case, the set of  $\Gamma$-supercyclic vectors is residual. However the converse is false. Indeed, if $X$ is Banach and $\Gamma$ is bounded then no homogeneous polynomial is $\Gamma$-transitive.

\begin{proposition}\label{Gamma transitividad}
Let $X$ be a separable Fr\'echet space and $\Gamma\sub X$. If $P\in\PP(^m X;X)$ is $\Gamma$-transitive then $P$ is $\Gamma$-supercyclic. 	
\end{proposition}
\begin{proof}
	Let $\{V_i\}_{i\in\zN}$ be a basis for $X$. For each $\gamma\in \Gamma$ and each $k$ choose $\gamma^\frac{1}{m^k}$ to be a $m^k$-root of $\gamma$. Consider for all $i$, 
	$$G_i:=\bigcup_{k,\gamma\in\Gamma} \gamma^\frac{1}{m^k} P^{-k} (V_i).$$
	Since $P$ is $\Gamma$-transitive, $G_i$ is a dense open set for each $i$. Therefore, $G=\bigcap_{i\in\zN} G_i$ is non empty. Finally we notice that any point in $G$ is a  $\Gamma$-supercyclic vector.
\end{proof}

We propose the following criterion to study $\Gamma$-transitivity. 

\begin{definition}[$\Gamma$-transitivity criterion]\label{criterion}
	We say that $P$ satisfies the supercyclity criterion for $(\lambda_k)$ provided that there are dense sets $X_0$ and $Y_0$ and applications $S_k(x):Y_0\rightarrow X$ such that
	for all $x\in X_0$, $y\in Y_0$,
	
	i) $S_{k}(x)(y)\rightarrow 0$ and
	
	ii) $\lambda_k P^{n_k}(x+S_k(x)(y))\rightarrow y$.
\end{definition}
\begin{proposition}\label{prop criterio super}
	Let $(X,(\rho_n)_n)$ be a separable Fr\'echet space, where $(\rho_n)_n$ is a fundamental system of nondecreasing seminorms,  and let $P$ be a homogeneous polynomial. Then $P$ satisfies the $\Gamma$-transitivity criterion with to respect some $(\lambda_k)\sub \Gamma$ if and only if $P$ is $\Gamma$-transitive.
\end{proposition}
\begin{proof}
	Suppose that $P$ satisfies the $\Gamma$-transitivity	 criterion with respect to $(\lambda_k)$. Let $U,V$ be open sets. Let $x\in U\cap X_0$ and $y\in V\cap Y_0$. Since $S_k(x)(y)\to 0$  there is some with $x+S_k(y)\in U$ for large $k$. Since $\lambda_k P^{n_k}(x+S_k(x)(y))\to y$ there is some $k$ such that  $P^{n_k}(x+S_k(x)(y))\in V$.
	
	Suppose now that $P$ is $\Gamma$-transitive. Let $(x_i)_{i\in\zN}$ be a dense numerable sequence in $X$. Let $X_0=Y_0=\{x_i:i\in\zN\}$. 
	Fix $x,y\in X_0$. 
	For each $k\in\mathbb N$ there is some $z_k\in \{w:\,\rho_k(x-w)<\f{1}{k}\}$ such that for some $n_k>n_{k-1}$ and $\lambda_k\in \Gamma$ we have $\lambda_kP^{n_k}(z_k)\in \{w:\,\rho_k(y-w)<\f{1}{k}\}$ . We define $S_k(x)(y)=z_k-x$. 
	Since $z_k\to x$, $S_k(x)(y)\to 0$ while $\lambda_kP^{n_k}(x+S_k(x)(y))\rightarrow y$.
\end{proof}
\begin{example}\label{unbounded superciclico}
 Let $X$ be $c_0$ and $P:c_0\rightarrow c_0$ as $P(a)_j:=a_{j+1}^m$. The polynomial $P$ is $\Gamma$-supercyclic for every unbounded $\Gamma\sub\C$.
\end{example}
\begin{proof}
We will apply Criterion \ref{criterion}.
Let $X_0=Y_0=c_{00}$ and let $x,y\in c_{00}$. 
   We now define $S_k(x)(y)$. For $\lambda\in\C$ let $\lambda^\frac{1}{m}$ be any $m$-root of $\lambda$. Let  $F:c_{00}\rightarrow c_{00}$ be defined  as
$$F(a)_j:= \left\{ \begin{array}{cc}
                     a_{j-1}^\frac{1}{m}&\text{ if } j\neq 1\\
                     0                      & \text{ if } j=1      ,                
                         \end{array}\right.  $$
 and set $S_k(x)(y)=\frac{F^k(y)}{\lambda_{k}^\frac{1}{m^k}}$, where $\lambda_{k}\in\Gamma$ is chosen so that  $\lambda_{k}^\frac{1}{m^k}$ tends to infinity. 
 
 Since $\|F(y)\|\leq \max\{1,\|y\|\}$ we get $S_k(x)(y)\to 0$. On the other hand, if $k>  \max \{j:x_j\neq 0\}$ it follows that 
 $$\lambda_{k}P^k(x+S_k(x)(y))=\lambda_{k}P^kS_k(x)(y)=P^kF^k(y)=y.$$
 By Proposition \ref{prop criterio super}, $P$ is $\Gamma$-transitive and hence $\Gamma$-supercyclic.
 \end{proof}
 The above  polynomial cannot provide an example of a $\Gamma$-supercyclic homogeneous polynomial when $\Gamma$ is a bounded set. Indeed, if $\Gamma$ is bounded, then by Proposition \ref{Gamma vectors} the $\Gamma$-supercyclic vector is not in $A_P$, but in this case $A_P=c_{0}$. The same happens if we consider this polynomial on $\ell_p$, $1\le p<\infty$. However, in the space $c$ of convergent sequences the situation is different.
    \begin{example}\label{D-superciclico}
       Let $X=c$ and $P:c\rightarrow c$ defined as 
        $P(a)_j=(a_{j+1})^m.$ Then $P$ is $\zD$-supercyclic.
   \end{example}
 
   \begin{proof}
   For $A\sub \zN$ we denote by $\mathbb 1_A$ the characteristic vector of $A$, i.e.,   
   $$\mathbb 1_A(j)=\left\{\begin{array}{cr}
                                         1 & \text{ if } j\in A;\\
                                         0 & \text{ else }.
                                       \end{array}\right.$$

%
  For each $\lambda=re^{i\theta}\in\C$ let $\lambda^{\frac{1}{n}}$ denote the number $r^{\frac{1}{n}}e^{i\theta/n}$.

 We define $F:c\rightarrow c$ as 
 $$F(a)_j:=\left\{ \begin{array}{cr}
 a_{j-1}^\frac{1}{m}& \text{ if } j>1;\\
 0                  & \text{ if } j=1.
 \end{array}\right.$$                                    
 Clearly $F$ is a well defined map and $P^kF^k=Id,$ for every $k\in\zN$.
 
  Let $ (\tilde y_n)_n$ be a dense sequence in $c$ such that $ (\tilde y_n)_n=(y_n)_n + (l_n\mathbb 1_{sop(y_n)^c})_n$, where $(y_n)_n$ is a dense sequence in $c_{00}$, and $(l_n)_n$ is a sequence in $\C$.  Without loss of generality, we may suppose that, for all $n$,  $l_n\neq 0$ and $supp(y_n)=[1,m(y_n)]$, where  $m(y_n):=\max\{j:[y_n]_j\neq 0\}$. 
  
  We can construct, by induction, a sequence $(n_k)_k\sub\zN$ satisfying the following properties:
  
  i) $n_k>m(y_{k-1})+ n_{k-1}$, 
  
  ii) for all $j\in n_k + supp(y_k)$, $\left|\frac{[F^{n_k}(y_k)]_j}{l_k^\frac{1}{m^{n_k}}}-1\right|<\frac{1}{k}$, 
   
  iii) for every $k>0$, $\frac{l_k}{2^{m^{n_k}}}< 1$ and 
   
  iv)  for all $j<k$, and all $i\in (n_k-n_j)+supp(y_k)$, $\left|l_j\left(\frac{[F^{n_k-n_j}(y_k)]_i}{l_k^\frac{m^{n_j}}{m^{n_k}}}-1\right)\right|<\frac{1}{k}.$
  
 Note that it is possible to choose such a sequence $(n_k)_k$ because all four conditions will be satisfied once $n_k$ is sufficiently large.
 
 Let $\tilde x\in  \ell_\infty$ be the vector defined as
 $$
 \tilde x=\sum_{j=1}^\infty  \frac {2 F^{n_j}(y_j)}{l_j^{\frac{1}{m^{n_j}}}}.
 $$
The vector $\tilde x \notin c$ since there are gaps with zeros. By condition ii) the nonzero coordinates of $\tilde x$ tend to $2$. Therefore we fill the gaps to obtain a well defined vector in $c$, we define
 $$
 x=\tilde x + 2\mathbb 1_{supp (\tilde x)^c}.
 $$
     By condition i), the supports of the $F^{n_j}(y_j)$ are pairwise disjoint. Consequently, 
  \begin{equation}\label{P(x)}
   P^k(\tilde x)=\sum_{j=1}^\infty P^k\left( \frac {2F^{n_j}(y_j)}{l_j^{\frac{1}{m^{n_j}}}}\right).
  \end{equation}
 Also, by condition i), it follows that if $k>j$, then
 \begin{equation}\label{PF=0}
  P^{n_k}(F^{n_j}(y_j))=0.
 \end{equation}
  We claim that $x$ is a $\zD$-supercyclic vector. Indeed, let $k$ be in $\zN$. By condition iii) each $\frac{l_k}{2^{m^{n_k}}}\in\zD$ and thus,
  \begin{eqnarray}\label{eqDsuper}
   \nonumber \left\|\frac{l_k}{2^{m^{n_k}}}P^{n_k}(x)-\tilde y_k\right\|&=&\left\|\frac{l_k}{2^{m^{n_k}}}P^{n_k}\left(\tilde x+2\mathbb 1_{supp(\tilde x)^c}\right) - y_k-l_k\mathbb 1_{supp(y_k)^c}\right\|\\
    &=&\left\|\frac{l_k}{2^{m^{n_k}}}P^{n_k}(\tilde x)+\frac{l_k}{2^{m^{n_k}}}2^{m^{n_k}}P^{n_k}\left(\mathbb 1_{supp(\tilde x)^c}\right) - y_k-l_k\mathbb 1_{supp(y_k)^c}\right\|.
  \end{eqnarray}
  Let us observe that $\zN=supp(P^{n_k}(\tilde x +\mathbb 1_{supp(\tilde x)^c}))=supp(P^{n_k}(\tilde x))\cup supp(P^{n_k}(\mathbb 1_{supp(\tilde x)^c})).$
  and that $P^{n_k}\left(\mathbb 1_{supp(\tilde x)^c}\right)=\mathbb 1_{supp(P^{n_k}(\tilde x))^c}$. Moreover, since $supp(y_k)\sub supp(P^{n_k}(\tilde x))$, we have also that
 $$\mathbb 1_{supp(P^{n_k}(\tilde x)^c)}-\mathbb 1_{supp(y_k)^c}=-\mathbb 1_{(supp(P^{n_k}(\tilde x))\setminus supp(y_k))}.$$
 Therefore, by equality \eqref{P(x)}, the above remarks and equation \eqref{eqDsuper},
 \begin{eqnarray}\label{eqDsuper2}
 \nonumber\left\|\frac{l_k}{2^{m^{n_k}}}P^{n_k}(x)-\tilde y_k\right\|&=&\left\|\frac{l_k}{2^{m^{n_k}}}P^{n_k}(\tilde x)+l_k\mathbb 1_{supp(P^{n_k}\tilde x)^c} - y_k-l_k\mathbb 1_{supp(y_k)^c}\right\|\\
 \nonumber&=&\left\|\frac{l_k}{2^{m^{n_k}}}P^{n_k}(\tilde x)-y_k-l_k\mathbb 1_{(supp(P^{n_k}(\tilde x))\setminus supp(y_k))}\right\|\\
 \nonumber&=&\left\|\frac{l_k}{2^{m^{n_k}}}P^{n_k}\left(\sum_{j=1}^\infty  \frac {2F^{n_j}(y_j)}{l_j^{\frac{1}{m^{n_j}}}}\right)-y_k-l_k\mathbb 1_{(supp(P^{n_k}(\tilde x))\setminus supp(y_k))}\right\|\\
 \nonumber&=&\left\|\sum_{j=1}^\infty \frac{l_k}{2^{m^{n_k}}}P^{n_k}\left(  \frac {2F^{n_j}(y_j)}{l_j^{\frac{1}{m^{n_j}}}}\right)-y_k-l_k\mathbb 1_{(supp(P^{n_k}(\tilde x))\setminus supp(y_k))}\right\|\\
 \nonumber&=& \left\|\sum_{j=1}^\infty\frac{l_k}{l_j^{\frac{m^{n_k}}{m^{n_j}}}}P^{n_k}\left(F^{n_j}(y_j)\right)-y_k-l_k\mathbb 1_{(supp(P^{n_k}(\tilde x))\setminus supp(y_k))}\right\|\\
 \nonumber&\leq&\left\|\sum_{j<k}\frac{l_k}{l_j^\frac{m^{n_k}}{m^{n_j}}}P^{n_k}F^{n_j}(y_j)\right\|+\left\|\frac{l_k}{l_k^\frac{m^{n_k}}{m^{n_k}}}y_k-y_k\right\|\\
 \nonumber&+&\left\|\sum_{j>k}\frac{l_k}{l_j^\frac{m^{n_k}}{m^{n_j}}}\left(F^{n_j-n_k}(y_j)\right)-l_k\mathbb 1_{(supp(P^{n_k}(\tilde x))\setminus supp(y_k))}\right\|.
 \end{eqnarray}
 By equality \eqref{PF=0} the first term of the last expression is zero. Notice also that $supp\left(\sum_{j>k}F^{n_j-n_k}(y_j)\right)=supp(P^{n_k}(\tilde x))-supp(y_k)$. Therefore, by property iv),
 \begin{align*}
  \left\|\sum_{j>k}\frac{l_k}{l_j^\frac{m^{n_k}}{m^{n_j}}}\left(F^{n_j-n_k}(y_j)\right)-l_k\mathbb 1_{(supp(P^{n_k}(\tilde x))\setminus supp(y_k))}\right\|
  &=\sup_{j>k}\sup_{i\in n_j-n_k+supp(y_j)}
  \left\{\left|l_k \left(\frac{[F^{n_j-n_k}(y_j)]_i}{l_j^\frac{m^{n_k}}{m^{n_j}}}-1\right)\right|\right\}\\
  &\leq \sup_{j>k} \frac{1}{j}=\frac{1}{k+1}\to 0.
 \end{align*}
  This implies that $x$ is a $\zD$-supercyclic vector.
 \end{proof}
 In Example \ref{J_P de c} we proved that the Julia set of the above polynomial is $J_P=\{x\in c: \lim x_j\in \zT\}$. Since for each $\lambda\in\zD$ and each $x\in J_P$, $\lim \lambda x_j\in \zD$, it follows that if $x$ is a $\zD$-supercyclic vector then $x\in R_P$. 
 \subsection{Main example.}
We proceed with the main example of our work. We prove in this subsection that the 2-homogeneous polynomial  $P:\ell_p\to \ell_p$ defined as $P=e_1'\cdot B,$ where $B$ is the backward shift operator,  is weakly hypercyclic, $d$-hypercyclic for every $d>r_P$ and $\Gamma$-supercyclic for every $\Gamma\subset \mathbb C$ such that $\Gamma$ is  unbounded or  not bounded away from zero. The proof is based on properties of its Julia set $J_P$.

\begin{theorem}\label{mainexample}
 Let $X=c_0$ or $X=\ell_p$, $1\le p<\infty$, and let $P:X\rightarrow X$ defined as
 $$P(x)=x_1 B(x),$$
 where $B: X\rightarrow X$ is the backward shift.
 Then $P|_{J_P}$ is chaotic and, if $x_0\in J_P$ is any vector such that $Orb_P(x_0)$ is dense in $J_P$ then 
 \begin{enumerate}
 	\item $x_0$ is a weakly hypercyclic vector, 
 	\item $x_0$ is a $d$-hypercyclic vector, for every $d>r_P$, 
 	\item $x_0$ is a $\Gamma$-supercyclic vector, for every $\Gamma\subset \mathbb C$ such that 0 is an accumulation point of $\Gamma$.
 	 \end{enumerate}
 Moreover,  $P$ is also $\Gamma$-supercyclic for every unbounded set $\Gamma\subset \mathbb C$.
 \end{theorem}
 We will prove the theorem for $X=\ell_p$, the proof for $X=c_0$ is similar.
 We will divide the proof in several steps. We will prove first that $P|_{J_P}$  is chaotic and then that the Julia set $J_P$ is $d$-dense, weakly dense and that $\overline{\Gamma\cdot J_P}=\ell_p$ for every $d>r_P$ and every $\Gamma$ with zero as an accumulation point. A dense orbit in $J_P$ inherits all the mentioned properties.
 
  Since $\Gamma$-supercyclic vectors must belong to $A_P$ if $\Gamma$ is bounded away from zero,  the case when $\Gamma$ is unbounded (but not necessarily with 0 as accumulation point) will be treated separately.
 
 We start by computing the norm of $P$. The problem is equivalent to maximizing the function $f(x,y)=xy$ under the restrictions $x^p+y^p=1, x\geq 0, y\geq 0$. Applying Lagrange multipliers it follows that $\|P\|=\frac{1}{2^{\frac{2}{p}}}$. Thus $r_P=2^{\frac{2}{p}}$.

 Notice that $P^n(x)=c_n(x)B^n(x)$, where
 $$c_n(x)=x_1^{2^{n-1}}\cdot x_2^{2^{n-2}}\cdot\ldots \cdot x_n;$$
  equivalently $c_n(x)$ can be  recurrently defined by the relations 
  $$\begin{cases}
  c_1(x)=x_1\\
  c_{n+1}(x)=c_{n}^2(x) x_{n+1}.
  \end{cases}$$  
   
 Since $c_{00}\sub A_P$, $A_P$ is dense and $R_P=\emptyset$. By Proposition \ref{completely invariant} $J_P$ must be completely invariant. If $x\in J_P$ and $|t|>1$ then $tx\notin A_P$. Since $R_P=\emptyset$, $tx$ must belong to $J_P$. Thus, we may think $J_P$ as an infinite union of half lines.

Let us first check that $J_P\neq \emptyset$. The following is a fixed vector for $P$
$$\overline x=2^\frac{1}{p}\l1,\frac{1}{2 ^\frac{1}{p}},\frac{1}{2 ^\frac{2}{p}},\frac{1}{2 ^\frac{3}{p}},\ldots \r.$$
Since $\overline x$ is a fixed vector, then it must, by Proposition \ref{estoy en J_P}, belong to $J_P$. Also, $\overline x\in r_PS_X$ since
\begin{align*}
\|\overline x\|_p^p&=2 \sum_{i=0}^\infty \l\f{1}{2^{\frac{1}{p}}}\r^{ip}=2\cdot \f{1}{1-\f{1}{2}}=4.
\end{align*}
   
\begin{proof}[Proof that $P|_{J_P}$ is chaotic]
  
 Since $J_P$ is closed and $P$-invariant, by Birkhoff's transitivity Theorem $P|_{J_P}$ is chaotic if and only if $P|_{J_P}$ is transitive and the periodic vectors of $P$ are dense in $J_P$.
  
 Let $U,V$ be nonempty open sets intersecting $J_P$. Let $y\in V\cap J_P$, $x\in U\cap J_P$. Since $B^n(x)\rightarrow 0$ and $\|P^n(x)\|=|c_n(x)|\|B^n(x)\|\geq r_P$
 we have that $c_n(x)\rightarrow \infty$.
 
 We will perturb $x$ to a vector $\tilde x$ in such a way that $\tilde x\in U$ and for some $n$, $P^n(\tilde x)=y$. Notice that, since $J_P$ is completely
 invariant, this implies that $\tilde x \in J_P$. Consider $x^n$ the vector $x^n=\sum_{i=1}^n e_i'(x)e_i$.
  
 Clearly $x^n\rightarrow x$ in $\ell_p$. Since $c_n(x)$ reads only the first $n$ coordinates, $c_n(x)=c_n(x^n)$.
 Consider $\tilde x^n=x^n+\frac{S^n(y)}{c_n(x)}$, where $S$ denotes the forward shift. Since $\|S^n(y)\|_p=\|y\|_p$ and $c_n(x)\rightarrow \infty$, it follows that
 $\frac{S^n(y)}{c_n(x)}\rightarrow 0$ and that  $\tilde x^n\to x$.
 Thus, for large $n$, $\tilde x^n\in U$. Also,
 \begin{align*}
  P^n(\tilde x^n)&=c_n(\tilde x^n)B^n\l x^n+\frac{S^n(y)}{c_n(x)}\r\\
  &=c_n(x)\frac{y}{c_n(x)}=y.
 \end{align*}  
 Therefore we may take  $\tilde x=\tilde x^n$ for $n$ sufficiently large.

 Next we show that the periodic vectors are dense in $J_P$. Let $y\in J_P$ and $\epsilon>0$. Since $y\in J_P$, there is some $n_0$ with $|c_n(y)|>1$ for every $n\geq n_0$. Consider $y^n=\sum_{i=1}^n e_i'(y)e_i$. The vector
 $$
 \tilde y^n=\sum_{i=0}^\infty \frac{S^{ni}(y^n)}{c_n(y)^i}=y^n+\sum_{i=1}^\infty \frac{S^{ni}(y^n)}{c_n(y)^i}
 $$
 is $n$-periodic for $P$. Indeed 
 \begin{align*}
 P^n(\tilde y^n)&=c_n(\tilde y^n) B^n\l\sum_{i=0}^\infty \frac{S^{ni}(y^n)}{c_n(y)^i}\r
 =c_n(y)\sum_{i=0}^\infty \frac{B^nS^{ni}(y^n)}{c_n(y)^i}\\
 &=\sum_{i=1}^\infty \frac{S^{n(i-1)}(y^n)}{c_n(y)^{i-1}}=\tilde y^n.
 \end{align*}
 If $n\geq n_0$ it is also well defined, because
 \begin{align*}
 \|\tilde y^n\|_p&\leq \sum_{i=0}^\infty \frac{\|S^{ni}(y^n)\|_p}{|c_n(y)|^i}\leq \sum_{i=0}^\infty \frac{\|y\|_p}{|c_n(y)|^i}<\infty. 
 \end{align*}
 Finally $\tilde y^n\to y$, since
 \begin{align*}
 \|\tilde y^n -y^n\|_p&\leq  \|y\|_p\sum_{i=1}^\infty \frac{1}{|c_n(y)|^i}\to 0.
 \end{align*}  
 \end{proof}

\begin{proof}[Proof that $J_P$ is $d$-dense, $d>r_P$]
Let $x\in \ell_p$. Recall that the vector $\overline x=2^\frac{1}{p}\l1,\frac{1}{2 ^\frac{1}{p}},\frac{1}{2 ^\frac{2}{p}},\frac{1}{2 ^\frac{3}{p}},\ldots \r$ belongs to $J_P$ and has norm equal to $r_P$.
Denoting $sign (0)=1$ we define $\tilde x$ as
$$\tilde x_i=sign (x_i) \max \{|x_i|,|\overline x_i|\}.$$
For each $i$ we have that $|\tilde x_i|\geq |\overline x_i|$. This implies that for each $n$, $\|P^n(\tilde x)\|_p\geq \|P^n(\overline x)\|_p\geq r_P$ and hence the vector belongs to $J_P$. Also $\|x-\tilde x\|_p\leq \|\overline x\|_p=r_P$.
\end{proof}
\begin{proof}[Proof that $J_P$ is weakly dense] Let $U=U_{\{\epsilon,x_0,\varphi_1,\ldots \varphi_n\}}$ be a basic weakly open set.
 
 Let $0\neq y\in \bigcap_{i=1}^n \ Ker(\varphi_i)$.  Let $M=\max\{\|\varphi_i\|\}$ and 
$\tilde\epsilon=\f{\epsilon}{M}$. We consider for each $n\in\zN$, $y^n=n\cdot y+x_0$. The vector $y^n$ may fail to belong to $J_P$, however we can find a perturbation $x^n$ of $y^n$ so that $x^n\in J_P\cap U$.
Noting again $sign(0)=1$ and considering the fixed vector $\overline x$, we define
 
 $$x^n_j:= sign(y^n_j) \max\{|y^n_j|,\f{\tilde \epsilon}{r_P}|\overline x_j|\}.$$
 
 Clearly $\|x^n-y^n\|_p\leq \f{\tilde\epsilon}{r_P}\|\overline x\|_p=\tilde\epsilon$.  We claim that for large $n$, $x^n\in J_P$. Let $k$ such that $y_k\ne 0$,  so that $0\neq x^n_k\to \infty$ as $n\to\infty$.
 
 The vector
 $$
 z^n=\frac{\tilde\epsilon}{r_P} \l\overline x_1,\overline x_2,\ldots,\overline x_{k-1},x^n_k,\overline x_{k+1},\ldots \r
 $$
 satisfies that $|x^n_j|\geq |z^n_j|$ for every $j\geq 0$ and hence if $z^n$ belongs to $J_P$ so does $x^n$. But 
 $$
 P^k(z^n)= \l\frac{\tilde\epsilon}{r_P}\r^{2^k}\f{c_k\l(\overline x_1,\overline x_2,\ldots,\overline x_{k-1},x^n_k,\overline x_{k+1},\ldots) \r}{c_k(\overline x)}P^k (\overline x)
=\l\frac{\tilde\epsilon}{r_P}\r^{2^k}\f{x^n_k}{\overline x_k}P^k (\overline x)
=\left( \frac{\tilde\epsilon}{r_P} \right) ^{2^k}\f{x^n_k}{\overline x_k} \overline x.
 $$
  Thus for big enough $n$ we get that $P^k(z^n)=\lambda \overline x$ for some $|\lambda|\ge1$ and therefore $P^k(z^n)\in J_P$ for large $n$. Since $J_P$ is completely invariant this implies that $z^n\in J_P$.
 
 Finally $x^n$ belongs to $U$ because,
 $|\varphi_i (x^n-x_0)|
 \le \|\varphi_i\| \|x^n-y^n\|_p+ |\varphi_i(y^n-x_0)|\leq  \epsilon$.
 \end{proof}
 \begin{proof}[Proof that $\Gamma\cdot J_P$ is dense for every $\Gamma$ with zero as accumulation point]
Let $x\in\ell_p ,$ $\epsilon>0$ and $\gamma\in \Gamma$  such that $\f{\epsilon}{r_p\gamma}>1$.
We consider
$$z=\frac{1}{\gamma}sign (x_j)\max \{|x_j|,\f{\epsilon}{r_P}\overline x_j\}.$$
The vector $z$ is in $J_P$ since for each coordinate $|z_j|>|\overline x_j|$ and $\overline x_j\in J_P$. It is also immediate that $\|\gamma z-x\|_p<\epsilon.$
 
\end{proof}
\begin{proof}[Proof that $P$ is $\Gamma$ supercyclic for every unbounded $\Gamma$]
We will apply Criterion \ref{criterion}. Let $X_0=Y_0$ be a dense set of vectors with nonzero coordinates. For $x,y\in X_0$ fixed we choose $(\lambda_n)_n\subset \Gamma$ such that $\lambda_nc_n(x)\to \infty$. Let $x^n$ be the vector $x^n_j= \chi_{[1,n]}(j) x_j$, i.e. $x^n$ is the truncation of $x$. Finally we define the inverses $F_n(x)$ as $F_n(x)(y)= 
x^n-x+\f{S^{n}(y)}{c_n(x)\lambda_n}$, where $S$ is the forward shift operator.
It is clear that $F_n(x)(y)\to 0$. Also
\begin{align*}
\lambda_n P^n(x+F_n(x)(y))&=\lambda_n P^n\l x^n+\f{S^{n}(y)}{c_n(x)\lambda_n}\r=
\lambda_n c_n\l x^n+\f{S^{n}(y)}{c_n(x)\lambda_n}\r B^n\l x^n+\f{S^{n}(y)}{c_n(x)\lambda_n}\r  \\
&=
\lambda_n c_n(x)\f{y}{c_n(x)\lambda_n}
=y.
\end{align*} 
	\end{proof}

We end the example by pointing some comments.
The first to notice  is that we can arrive to $d$-dense orbits of arbitrary small radius.
\begin{remark}
For every $d>0$ there is some $d$-hypercyclic homogeneous polynomial.
\end{remark}
\begin{proof}
Let $P$ be the 2-homogeneous polynomial from the above theorem. By Proposition \ref{quasiI}, $\lambda P$ is quasiconjugated to $P$ for every $\lambda\neq 0$ under a linear isomorphism of norm $\f{1}{|\lambda|}$. By Proposition \ref{d hiper}, this implies that $\lambda P$ is $d$-hypercyclic for every $d>r_{\lambda P}=\f{r_P}{|\lambda|}$. 
\end{proof}

 \begin{remark}
  Proposition \ref{seg} says that $tx$ with $|t|<1$ is not   an accumulation point of $Orb_P(x)$. The example from Theorem \ref{mainexample} shows that this does not happen for $|t|\ge1$.  Indeed, if $x$ is a point whose $P$-orbit is dense in ${J_P}$, $tx$ is in $J_P$ for $|t|\ge1$ and hence it is an accumulation point of the $P$-orbit of $x$. 
 \end{remark}
 
 \begin{remark}
 The Julia set of the polynomial of Theorem \ref{mainexample} satisfies another nice property: the image of $J_P$ under the backward shift is dense in $\ell_p$. This implies that the dense orbit in $J_P$ must satisfy that $\overline{\{B(P^n(x))\}}=\ell_p$. Thus, the family $\{c_n(\cdot)B^{n+1}(\cdot)\}$ is an universal family of homogeneous polynomials while $\{c_n(\cdot)B^n(\cdot)\}=\{P^n\}$ has only nowhere dense orbits.  
 \end{remark}
 \begin{proof}[Proof that $\overline{B(J_P)}=\ell_p$]
 Let $\epsilon>0$ and $x\in\ell_p$. Let $\overline x$ be a fixed vector that satisfies $\|\overline x\|_p=r_P$.
 Define $y\in \ell_p$ as $y_n=sign(x_n)\max\{|x_n|,\epsilon|\overline x_n|\}$. Thus $\|y-x\|_p\le \epsilon r_P$ and $|y_n|\ge \epsilon|\overline x_n|$ for each $n$. We define $z=\frac{e_1}{\epsilon}+S(y)$. Then $B(z)=y$ and $z\in J_P$. Indeed, $P(z)=\frac{y}{\epsilon}$ which satisfies $|\frac{y_n}{\epsilon}|>|\overline x_n|$ for each coordinate $n$. Thus, $P(z)$ belongs to $J_P$  and therefore $z$ is also in $J_P$. 
 \end{proof}
Recall that a dynamical system $(X,F)$ is said to be Devaney chaotic if it is  transitive, the periodic vectors are dense  and if it has \textit{sensitive dependence on the initial conditions}. This last condition means that 
there exists a neighborhood $U$ of 0 such that for every $x\in X$ and every neighborhood $V$ of 0, there is some $n\in\zN$ and $y\in x+V$ such that $F^n(y)\notin U+F^n(x)$.
On metric spaces, sensitivity on the initial conditions is implied by the transitivity and the density of periodic vectors and, of course, no nonlinear homogeneous polynomial on a Banach space can have sensitivity on the initial conditions. But, if we consider the weak topology on the Banach space, we may have Devaney chaotic polynomials. Indeed, if $P$ is the polynomial from Theorem \ref{mainexample}, we already proved that $P|_{J_P}$ satisfies that it is transitive and that the periodic vectors are dense, and moreover we saw that $J_P$ is weakly dense in $\ell_p$. Thus, $P$ satisfies the first two conditions of Devaney chaoticity. To see the last one, just take $U=\{x\in\ell_p:\,|x_1|<1\}$. Then for $x\in\ell_p$, $V$ a weak neighborhood of 0, which we may suppose that, for some $\epsilon>0$ and some $n$,  $V\supset \{x\in\ell_p:\,|x_j|<\epsilon,\, j=1,\dots,n\}$. define $y\in\ell_p$ as,
$$
y_k=\left\{\begin{array}{ll}
	\epsilon/2 & \text{ if }k\le n,\, x_k=0,\\
	x_k & \text{ if }k\le n,\, x_k\ne 0,\\
		c & \text{ if }k= n+1,\\
	0 & \text{ if }k> n+1.
\end{array}\right. 
$$
Then $y\in x+V$ for every $c$ and 
$$
e_1'(P^nx-P^ny)=e_1'\Big(c_n(x)B^n(x)-c_n(y)ce_1\Big)= c_n(x)x_{n+1}-c_n(y)c.
$$
Since $c_n(y)\ne 0$ and it does not depend on $c$, we can find $c$ such that $|e_1'(P^nx-P^ny)|>1$ and thus $P^n(y)\notin U+P^n(x)$. Therefore $P=e_1'\cdot B$ is Devaney chaotic on $(\ell_p,w)$.

%

\begin{example}\label{no completamente inv}
	Using the polynomial $P$ defined in Theorem \ref{mainexample}, it is possible to exhibit an example of a homogeneous polynomial whose Julia set is not completely invariant: let $X=\mathbb C\oplus_\infty c_0(c_0)$ and define $Q\in \mathcal P(^2X)$ as 
	$$
	Q(\lambda,(y^j)_{j\in\mathbb N})=(0,P(\lambda x),P(y^1),P(y^2),\dots),
	$$
where $x\in c_0$ is a vector in the Julia set of $P$. Then, for $|\lambda|>1$, the vector $(\lambda, (0)_j)$ is in $R_Q$ because if $|t|>1$,
$$
\|Q^n(t,(y^j)_{j\in\mathbb N})\|_X\ge \|P^n(tx)\|_\infty\to\infty.
$$
On the other hand, the vector $Q(\lambda, (0)_j)=(0,P(\lambda x),0,0,\dots)$ belongs to $J_Q$. Indeed, we may approximate $P(\lambda x)$ by a vector $z\in c_0$ with only  finite nonzero coordinates. Then the vector $(0,z,0,0,\dots)$ approximates $(0,P(\lambda x),0,0,\dots)$ in $X$ and clearly, $(0,z,0,0,\dots)$ is in $A_Q$. 
\end{example}	
	
\subsection{Numerically hypercyclicity}
The notion of numerically hypercyclicity was recently introduced by Kim, Peris and Song \cite{KimPer12}. A function  $F$ is said to be numerically hypercyclic provided that there are vectors $x\in S_X,x^*\in S_{X^*}$, with $x^*(x)=1$ and such that its numerical orbit $Norb_F(x,x^*):=\{x^*(F^n(x)):\, n\in \mathbb N_0\}$ is dense in $\C$. In \cite{KimPer12,KimPer122} the authors proved that every infinite dimensional and separable Banach space supports a homogeneous polynomial of degree $d\geq 1$. In view of the Example \ref{mainexample} one may expect that the homogeneous polynomial $P(x)=x_1B(x)$ is numerically hypercyclic in $\ell_p$ or $c_0$. However since the limit radius of the polynomial is $r_P=2^\f{2}{p}>1$, for $X=\ell_p$  and $r_P=1$ for $X=c_0$, $\|P^n(x)\|\le r_P$ for every $x\in S_X$ and every $n$, and thus the numerical orbits  can never be dense.

The lack of numerically dense orbits is only due to the fact that $B_X\sub r_PB_X$. Thus, by considering a multiple of the polynomial and exploiting that homogeneous polynomials behave well under conjugation via a linear isomorphism we can easily construct a numerically hypercyclic homogeneous polynomial.
\begin{example}[A numerically hypercyclic homogeneous polynomial]
Let $X=\ell_p$ and $P(x)=x_1 B(x)$. Let $x_0\in J_P$ such that the orbit of $x_0$ under $P$ is dense in ${J_P}$ is hypercyclic.  Consider $\tilde P=\|x_0\| P$.  
By Proposition \ref{conjugar J_P} applied to $\Phi(x)=\f{1}{\|x_0\|}x$ it follows that $J_{\tilde P}= \|x_0\|J_P$ and that $\f{x_0}{\|x_0\|}$ is hypercyclic for $\tilde P|_{J_{\tilde P}}$. 
Now, by Example \ref{mainexample} and Proposition \ref{weakly conj}, $\f{x_0}{\|x_0\|}$ is a weakly hypercyclic vector for $\tilde P$. Thus, for any $x^*\in S_{X^*}$ it satisfies that $\overline{x^*({\tilde {P} }^n(x))}=\C$. 
\end{example} 


\begin{remark}
	There are several notions of numerical range for nonlinear mappings on Banach spaces, which coincide for linear operators (see for example \cite[Chapter 11]{appell2004nonlinear}).
 The numerical range in the sense of Bauer of a mapping $F$ is the set 
	$$
\{x^*(F(x)): x^*(x)=\|x\|^2,\, \|x\|=\|x^*\|\},
$$
i.e. the norm of $x$ and $x^*$ is allowed to be different from 1. Inspired in this notion of numerical range, a polynomial would be numerically hypercyclic if there are vectors $x\in X,x^*\in {X^*}$, with $x^*(x)=\|x\|^2$, $\|x\|=\|x^*\|$ and such that its numerical orbit $Norb_F(x,x^*):=\{x^*(F^n(x)):\, n\in \mathbb N_0\}$ is dense in $\C$. With this modified definition, the polynomial of Theorem \ref{mainexample} would also be numerically hypercyclic.
\end{remark}

\section{Existence of $d$-hypercyclic, weakly hypercyclic and $\Gamma$-supercyclic homogeneous polynomials on arbitrary Fr\'echet spaces}\label{Mainsection}

In this section we will prove our main Theorem \ref{main theorem}, which states that on any separable and infinite dimensional Fr\'echet space there exists a $d$-hypercyclic, weakly hypercyclic and $\Gamma$-supercyclic homogeneous polynomial for every subset $\Gamma$ which is unbounded or not bounded away from zero.

The   proof will follow some of the ideas of Bonet and Peris \cite{BonePer98} to prove the existence of hypercyclic operators on arbitrary Fr\'echet spaces. We will look for a polynomial $P$ acting on $\ell_1$ satisfying all the mentioned properties and at the same time we want to find on each Fr\'echet space $X$ a polynomial $Q$ being quasiconjugated to $P$.

The polynomial $P$ on $\ell_1$ will be of the type $e_1'(x)B_w(x)$, where $B_w$ is a weighted backward shift. The main difference with Bonet and Peris proof is that not every weight works. Indeed, if the weights are too small then the Julia set can be  empty. This forces us to adapt their main Lemma \cite[Lemma 2]{BonePer98} to our requirements.  

\begin{theorem}\label{Ejl1}
  Let $X=\ell_1$ and $P$ the homogeneous polynomial $P=e_1'\cdot B_{\omega}\in \PP(^2\ell_1;\ell_1)$, where $B_{\omega}$ is the weighted backward shift defined as $[B_\omega(x)]_i= \frac{1}{(i+1)^4}x_{i+1}$. Then $P$ is weakly hypercyclic, $d$-hypercyclic and $\Gamma$-supercyclic for every $\Gamma\subset \mathbb C$ such that $\Gamma$ is unbounded or zero is an accumulation point of $\Gamma$. 
\end{theorem}
The strategy of the proof is the same that we used in Example \ref{mainexample}. We will prove that $P|_{J_P}$ is hypercyclic and that $J_P$ is weakly dense, $d$-dense and $\overline{\Gamma J_P}=\ell_1$ for $\Gamma$ that accumulates at zero. Again we will analyze the case when  $\Gamma$ is unbounded  separately.

Note that $c_{00}\sub \bigcup_n Ker P^n$ so that $R_P=\emptyset$ and by Proposition \ref{completely invariant}, $J_P$ is completely invariant.  The fact that $R_P=\emptyset$ implies also that 
if $x\in J_P$ then all the ray $\{tx:|t|\ge 1\}\sub J_P$.

The iterations of a vector $x$ are 
$$
P^n(x)=\big(x_1^{2^{n-1}}\dots x_n\big)\cdot\Big( \frac{1}{1^{2^{n-1}}}\frac{1}{(2!^4)^{2^{n-2}}}\dots \frac{1}{((n-1)!^4)^{2}}\frac{1}{n!^4}\Big)\cdot B^n_{\omega}(x).
$$
So if we define $\Omega\in \C^\zN$ as $(\Omega)_n=\frac{1}{(n!)^4}$, and $c_n:\C^\zN\rightarrow \C$ as $c_n(z)=\prod_{i=1}^n z_i^{2^{n-i}}$,
then
$$P^n(x)=c_n(x)c_n(\Omega)B_\omega^n(x).$$

We must verify that the Julia is not empty. A typical member of $J_P$ must satisfy that the first coordinates are large in comparison
to the decreasing rate of the tail and the weights $\frac{1}{n!^4}$.
%
%

\begin{lemma}\label{en Jp}
	Let $j,l\in \zN$ and $k_j>1$ such that for all $n\in \zN$,
	$$
	k_j\l  2^{2^{\frac{n+1}{2}}}\r^{\sqrt 2-1}\geq (n+1)!^4(j+n-1)!^l.
	$$ 
	Then for every $N\geq k_j2^{2^{\frac{1}{2}}}$ the vector
	\begin{equation*}
	x=\l N,\frac{1}{j!^l}, \frac{1}{(j+1)!^l},...\r
	\end{equation*}
	belongs to $J_P$.
\end{lemma}
\begin{proof}
		We are going to prove by induction that for all $n$,
	$c_n(x)c_{n}(\Omega)\geq k_j 2^{2^{\frac{n}{2}}}$. By definition of $N$, $c_1(x)c_1(\Omega)=N\geq k_j 2^{2^\f{1}{2}}$. Suppose that our claim is true for some $n\ge1$, then,
	\begin{align*}
	c_{n+1}(x)c_{n+1}(\Omega)&
	=\frac{c_{n}(x)^2c_{n}(\Omega)^2}{(n+1)!^4 (j+n-1)!^l}
	\geq \f{k_j^2 2^{2^{\frac{n}{2}+1}}}{(n+1)!^4 (j+n-1)!^l}\\
	&=k_j2^{2^{\frac{n+1}{2}}}
	\f{k_j \l 2^{2^{\frac{n+1}{2}}}\r^{\sqrt 2-1}}{(n+1)!^4 (j+n-1)!^l} \ge k_j2^{2^{\frac{n+1}{2}}}.\\
	\end{align*}
	Thus, since $\|B_\omega ^n(x)\|\geq |B_\omega ^n(x)_1|= \f{1}{(n+1)!^4(j+n-1)!^l}$, it follows that $$\|P^n(x)\|\ge \f{k_j2^{2^{\frac{n+1}{2}}}}{(n+1)!^4(j+n-1)!^l}\to \infty,$$
	as $n\to \infty$, and therefore $x$ is in $J_P$.
\end{proof}

\begin{proof}[Proof that $P|_{J_P}$ is hypercyclic]
We prove that $P|_{J_P}$ is transitive. By Birkhoff's transitivity Theorem it follows that there exists a dense orbit in $J_P$.

Inspired in the hypercyclicity criterion for linear operators we will find two $J_P$-dense sets $X, \ Y$ and applications $S_{n,x}:Y\rightarrow \ell_1$ such that the following properties are satisfied:

i) for all $x\in X$, $y\in Y$, $S_{n,x}(y)\rightarrow 0$,

ii) if $x^n=\sum_{i=1}^n e_i'(x)e_i$, then $x^n+ S_{n,x}(y)\in J_P$ and

iii) for all $x\in X$, $y\in Y$, $P^n(x^n+S_{n,x}(y))\rightarrow y$. 

If those properties are satisfied then clearly $P|_{J_P}$ is transitive.

We define $X$ as $X:=\{x\in J_P: tx\in J_P \text{ for some } |t|<1\}$. Since $R_P$ is the empty set, if $x\in J_P$ then $tx\in J_P$ for all $|t|>1$. This implies that $X$ is dense in $J_P$.

For $y\in J_P$ we define $y^{(n)}\in \ell_1$ as
$$y^{(n)}_j:=\begin{cases}
                y_j&\text{ if } j\leq n;\\
                \frac{1}{j!}&\text{ if } j>n.
               \end{cases}$$
We define $Y=\{y^{(n)}:y^{(n)}\in J_P\text{ and } y\in X \}\subset X\subset J_P.$ We defined $Y$ in such a way because of two reasons. We want the formal inverse $S^n_\omega$ of $B^n_\omega$ to be defined in $Y$ for all $n$, but at the same time
we  want to have a (good) control of the decreasing rate of the tail.

Note that
\begin{equation}\label{coordenadas shift}
 S^k_\omega(y^{(n)})_j=\begin{cases}
                    0 &\text{ if } j\leq k;\\
                    \frac{j!^4}{(j-k)!^4}y_{j-k}&\text{ if } k< j\leq k+n;\\
                    \frac{j!^4}{(j-k)!^4 (j-k)!}&\text{ if } j> k+n.
                   \end{cases}      
\end{equation}
Thus $S^n_\omega(y^{(k)})\in\ell_1$ for all $k$.

In order to show that $Y$ is dense in $J_P$, it suffices to see that $Y$ is dense in $X$. 
Since $y^{(n)}\rightarrow y$, it suffices to show that if $y\in X$ then for large $n$, $y^{(n)}\in J_P$.
We will exploit  the fact that $c_n(y)$ grows very fast when $y\in X$. Let $|t|>1$ with $\frac{y}{t}\in J_P$. We have that there is some constant $C>0$ such that for every $n$,
\begin{equation}\label{c_n crece rapido}
 \left|c_n(y)c_n(\Omega)\right|=|t|^{2^n-1}\left|c_n\l\frac{y}{t}\r c_n(\Omega)\right|\geq C|t|^{2^n}.
\end{equation}
Since $P$ is completely invariant it suffices to check that $P^n(y^{(n)})=c_n(y)c_n(\Omega) (\f{1}{n!},\f{1}{(n+1)!},\ldots)$ belongs to $J_P$. 
The coordinates of this vector are, for sufficiently large $n$, greater than the coordinates of $( \f{c_n(y)c_n(\Omega)}{n!},\f{1}{(n+1)!},\ldots)$. Thus, it suffices to show that this new vector is in $J_P$ for large $n$. This is a straightforward application of Lemma \ref{en Jp} because $\f{t^{2^n}}{2(n+1)!^5}\to \infty$. Therefore, it follows that $y^{(n)}\in J_P$ for large $n$ and thus $Y$ is dense in $J_P$.

We are now able to define the inverse application $S_{n,x}(y)$. For $y\in Y$, $x\in X$, let $S_{n,x}(y):=\frac{S^n_{\omega}(y)}{c_{n}(x)c_n(\Omega)}$. Observe that $c_n$ reads only the first $n$ coordinates, so we get
\begin{align*}
 P^n\l x^n+\frac{S^n_{\omega}(y)}{c_{n}(x)c_n(\Omega)}\r&=c_n(x)c_n(\Omega)B_\omega^n\l x^n+\frac{S^n_{\omega}(y)}{c_n(x)c_n(\Omega)}\r\\
 &=y.
  \end{align*}
 Since $y\in J_P$ and $J_P$ is completely invariant it follows that $x^n+\frac{S^n_{\omega}(y)}{c_n(x)c_n(\Omega)}\in J_P$. So we have that properties ii) and iii) of our criterion are satisfied.
It remains to show that $\frac{S^n_{\omega}(y)}{c_{n}(x)c_n(\Omega)}\rightarrow 0$, but recall that for $x\in X$,   $c_n(x)c_n(\omega)|t|^{-2^n}\to \infty$ for some $|t|>1$. At the same time we have a good control of the decreasing rate of the tail of $y$.

Take $n$ such that $y=y^{(n)}$. Recall that by \eqref{coordenadas shift},
\begin{align*}
\|S^k_\omega(z^{(n)})\|&=\sum_{j=k+1}^{k+n} \frac{j!^4}{(j-k)!^4}|z_{j-k}|+ \sum_{j>{k+n}} \frac{j!^4}{(j-k)!^4! (j-k)!}\\
&\leq \|z\|\frac{(k+n!)^4}{n!^4}+\sum_{j>0} \frac{(j+k+n)!^4}{(j+n)!^5}.\\
\end{align*}
Note that  $\sum_{j>0} \frac{(j+k+n)!^4}{(j+n)!^5}$ is a convergent series. Let $|t|>1$ such that $\frac{x}{t}\in J_P$.  Then by  \eqref{c_n crece rapido}, we have 
\begin{equation}\label{periodic}
 \left\|\frac{S^k_{\omega}(y)}{c_{k}(x)c_k(\Omega)}\right\|\leq \frac{\|z\|\frac{(k+n!)^4}{n!^4}+\sum_{j>0} \frac{(j+k+n)!^4}{(j+n)!^5}}{Ct^{2^k}}\rightarrow 0,
\end{equation}
This proves that $P|_{J_P}$ is transitive.
\end{proof}



\begin{proof}[Proof that $J_P$ is $d$-dense]
	Let $x\in \ell_1$. Let $N$ such that the vector $\overline x=(N,\f{1}{2!},\f{1}{3!},\ldots)\in J_P$.
	We define $\tilde x$ as
	$$\tilde x_i=sign (x_i) \max \{|x_i|,|\overline x_i|\},$$
	with $sign (0)=1$.
	For each $i$ we have that $|\tilde x_i|\geq |\overline x_i|$. This implies that for each $n$, $\|P^n(\tilde x)\|\geq \|P^n(\overline x)\|\geq r_P$ and hence the vector belongs to $J_P$. Also $\|x-\tilde x\|\leq \|\overline x\|$.
	
	Therefore $J_P$ is $d$-dense for every $d>\|\overline x\|$.
\end{proof}

\begin{proof}[Proof that $J_P$ is weakly dense]
	 Let $U=U_{\{\epsilon,x_0,\varphi_1,\ldots \varphi_n\}}$ be a basic weakly open set.
	
	Let $0\neq y\in \bigcap_{i=1}^n \ Ker(\varphi_i)$.  Let $M=\max\{\|\varphi_i\|\}$ and 
	$\tilde\epsilon=\f{\epsilon}{M}$. We consider for each $n\in\zN$, $y^n=n\cdot y+x_0$. The vector $y^n$ may fail to belong to $J_P$, however we can find a perturbation $x^n$ so that $x^n\in J_P\cap U$.  
	Denoting  $sign(0)=1$ and  $\overline x$ as in the previous proof, we define
	
	$$x^n_j:= sign(y^n_j) \max\{|y^n_j|,\f{\tilde \epsilon}{\|\overline x\|}|\overline x_j|\}.$$
	
	Clearly $\|x^n-y^n\|\leq \f{\tilde\epsilon}{\|\overline x\|}\|\overline x\|=\tilde\epsilon$.  We claim that for large $n$, $x^n\in J_P$. Let $k$ such that $y_k\ne 0$. Note that this implies that $0\neq x^n_k\to \infty$ as $n\to\infty$.

	The vector
	$$z^n=\frac{\tilde\epsilon}{\|\overline x\|} (\overline x_1,\overline x_2,\ldots,\overline x_{k-1},x^n_k,\overline x_{k+1},\ldots)
	$$
	satisfies that $|x^n_j|\geq |z^n_j|$ for every $j\geq 0$ and hence if $z^n$ belongs to $J_P$ so does  $x^n$. But
	$$
	P^k(z^n)= \left(\frac{\tilde\epsilon}{\|\overline x\|}\right)^{2^k}\f{c_k\l(\overline x_1,\overline x_2,\ldots,\overline x_{k-1},x^n_k,\overline x_{k+1},\ldots) \r}{c_k(\overline x)}P^k (\overline x)
	=\left(\frac{\tilde\epsilon}{\|\overline x\|}\right)^{2^k}\f{x^n_k}{\overline x_k}P^k (\overline x).
	$$
Since $P^k (\overline x)\in J_P$ we have that $P^k(z^n)\in J_P$ for sufficiently large $n$. 
Thus, the complete invariance of $J_P$ implies that $z^n\in J_P$.
	
	Finally $x^n$ belongs to $U$ because,
	$|\varphi_i (x^n-x_0)|\leq \|\varphi_i\| \tilde\epsilon+ \varphi(y^n-x_0)<  \epsilon$.
\end{proof}
\begin{proof}[Proof that $\Gamma\cdot J_P$ is dense for every $\Gamma$ with zero as accumulation point]
	Let $x\in\ell_1 ,$ $\epsilon>0$ and $\gamma\in \Gamma$  such that $\f{\epsilon}{\|\overline x\|\gamma}>1$.
	We consider
	$$z=\frac{1}{\gamma}sign (x_j)\max \{|x_j|,\f{\epsilon}{\|\overline x\|}\overline x_j\}.$$
	The vector $z$ is in $J_P$ since for each coordinate $|z_j|>|\overline x_j|$ and $\overline x_j\in J_P$. It is also immediate that $\|\gamma z-x\|_1<\epsilon.$
\end{proof}
\begin{proof}[Proof that $P$ is $\Gamma$ supercyclic for every unbounded $\Gamma$]
	It suffices to prove that $P$ is $\Gamma$-transitive for every unbounded $\Gamma$.
	Let $U$ and $V$ be open sets. We want to show that for some $\gamma\in\Gamma$ 
	and $n\in\zN$, $\gamma P^n(U)\cap V\neq$
	$ \emptyset$.
	 Since $V\cap J_P$ and $U\cap J_P$ may be both the empty set, we will skip the Julia set $J_P$. However 
	our approach will be very similar to the one that we used to prove that $P|_{J_P}$ was transitive.
	
	The inverses $S_\omega^k$ are not defined on $\ell_1$. However, by perturbing the tail of the sequence as we did in \eqref{coordenadas shift}, we have that the set $Y:=\{y\in \ell_1: S^k_\omega (y)\in\ell_1 \text{ for every } k\}$ is dense in $\ell_1$.
	
	Let $x\in U$ such that $x_j\neq 0$ for every $j$, and $y\in Y\cap V$  and consider $x^n:=\sum_{j\le n}x_je_j.$
	Since $\Gamma$ is unbounded we can choose $(\gamma_n)_n\sub \Gamma$ with $\f{S_\omega ^n(y)}{c_n(x)c_n(\Omega)\gamma_n}\rightarrow 0$. 
	Therefore,
	$$x^n+\f{S_\omega ^n(y)}{c_n(x)c_n(\Omega)\gamma_n}\rightarrow x$$
	and since $c_n$ reads only the first $n$ coordinates,
	$$\gamma_n P^n\l x^n+\f{S_\omega^n(y)}{c_n(x)c_n(\Omega)\gamma_n}
	\r=\gamma_nc_n(x^n)c_n(\Omega)B_w^n\l\f{S_\omega^n(y)}{c_n(x)c_n(\Omega)\gamma_n}\r=y.
	$$ 
\end{proof}

To prove the existence Theorem \ref{main theorem}, we will use
a version of the well known Lemma \cite[Lemma 2]{BonePer98} due to Bonet and Peris. This Lemma was a key ingredient to prove the existence of hypercyclic operators on arbitrary separable infinite dimensional Fr\'echet spaces. The main difference of this version is that we provide a control of the asymptotic behavior of the sequence $\alpha(n)=x_n^*(x_n)$.

\begin{lemma}
	Let $X$ be an infinite dimensional separable Fr\'echet space not isomorphic to $\C^\zN$ and let $\alpha(n)$ be a  sequence such that $n\alpha(n)\to 0$. Then, there are sequences $(x_n)_n$ in $X$ and $(x_n^*)_n$ in $X^*$ such that
	\begin{enumerate}
		\item $x_n\to 0$ and $span\{x_n\}$ is dense in $X$,
		\item $\{x_n^*\}$ is equicontinuous and 
		\item $x_n^*(x_k)=\alpha(n)\delta_{n,k}$.
	\end{enumerate} 
\end{lemma}

The proof of the above Lemma follows \cite{BonePer98}, together with a careful choice of the elements $(x_k)_k$.
\begin{proof}
	Let $(\rho_n)_n$ be a non-decreasing fundamental system of seminorms generating the topology on $X$.
	Since $X$ is not isomorphic to $\mathbb C^\zN$, there is a dense subspace $M\sub X$ and a continuous norm $\|\cdot\|$ defined on $M$. Without loss of generality we may assume that $\rho_1|_M=\|\cdot\|$.
	
	Consider $(z_n)_n\subset S_M$, a linearly independent sequence in the sphere of $M$ such that its span is dense in $M$. 
	 Let $(\beta_n)_n$ be a nondecreasing sequence such that, {$\beta_1>1$}, $\beta_n\to\infty$ and such that $n\alpha(n)\beta_n\to 0$.
	
%
	 We claim that, by adding new seminorms to our fundamental system $\{\rho_n\}_n$ and new vectors $(w_n)_n\sub S_M$ to the sequence $(z_n)_n$, we may assume that $\rho_k(z_j)\leq \beta_k$ for all $k\geq j$.

	Indeed, we will construct a subsequence $(n_l)_l$ of natural numbers and consider $(\psi_n)_n$ the non-decreasing fundamental system for $X$ given by $\psi_n=\rho_l$ for $n_{l-1}\le n< n_l$, and a new sequence of linearly independent vectors $(w_n)_n$, satisfying that $z_{l+1}=w_{n_l}$. This pair $(\psi_n,w_n)$ will satisfy  $\psi_k(w_j)\leq
	\beta_k$ 
	for all $k\geq j$. Since $span \{z_n:n\in\zN\}$ is already
	dense, it will follow that $span \{w_n:n\in\zN\}$ is also dense and since the $\psi_n$'s are the $\rho_n$'s but repeated, the systems $\{\psi_n\}$ and $\{\rho_n\}$ will generate the same topological space.  
	We construct the pairs $(\psi_n,w_n)$ inductively. Put $\psi_1=\rho_1, w_1=z_1$ and recall that $\rho_1(z_1)=1<\beta_1$. Since $\beta_k\to\infty$, there exists $n_1$ such that $\rho_2(z_1)< \beta_{n_1}$ and $\rho_2(z_2)<\beta_{n_1}$. We define  $\psi_{n_1}$ as $\rho_2$, $w_{n_1}=z_2$ and for $k<n_1$ we define $\psi_k$ as $\rho_1$. Since $\rho_2$ and $\rho_1$ are continuous, there exist $w_2,\ldots w_{n_1-1}\in S_M$ near $z_1$ and linearly independent to $\{z_n:n\in\zN\}$ such that 
	$\rho_2(w_i)<\beta_{n_1}$ for all $2\leq i\leq n_1-1$ and such that  $\psi_k(w_j)=\rho_1(w_j)=1<\beta_1\le \beta_k$ for all $2\le j\leq k\leq n_1$. Suppose now that we have constructed $\psi_1,\ldots \psi_{n_l}$ non-decreasing seminorms and vectors $w_1,\ldots w_{n_l}\in S_M$ such that the set $ \{w_i:{i\leq n_l}\} \cup \{z_n:{n\in\zN}\}$ is linearly independent, $w_{n_l}=z_{l+1}$, each $\psi_k$ is some of the $\rho_n$'s, $\psi_{n_l}=\rho_{l+1}$ and $ \psi_k(w_j)<\beta_k$ for all $j\leq k\leq n_l$. 
	We proceed as we did in the first step. There exists $n_{l+1}$ such that $\rho_{l+2}(w_j)< \beta_{n_{l+1}}$ for all $j\leq n_l$ and $\rho_{l+2}(z_{l+2})<\beta_{n_{l+1}}$. We put for $n_l<k<n_{l+1}$, $\psi_k=\rho_{l+1}=\psi_{n_{l}}$, $\psi_{n_{l+1}}=\rho_{l+2}$ and $w_{n_{l+1}}=z_{l+2}$. Notice that for $n_l<k< n_{l+1}$ and $j\leq n_l$ we have that $\psi_k(w_j)=\psi_{n_{l}}(w_j)<\beta_
	{n_{l}}\leq \beta_k$. If $k=n_{l+1}$, then this last relation also holds by the definition of $n_{l+1}$. Finally we choose for $n_l+1\leq j< n_{l+1}$ vectors $w_{j}\in S_M$ close to $z_1$ such that $\psi_{n_{l+1}}(w_j)<\beta_{n_{l+1}}$, such that, for $j\le k<n_{l+1}$, $\psi_{k}(w_j)=\psi_{n_{l}}(w_j)<\beta_{n_{l}}$  and such that $\{w_j:j\leq n_{l+1}\}\cup \{z_n:n\in\zN\}$ is linearly independent. The sequences $(\psi_k)_k$, $(w_j)_j$ constructed satisfy all the desired properties, and we have proved the claim.

	Now we choose by Hahn-Banach elements $(x_k^*)_k$ in $M^*$ such that $x_k^*(z_j)=\delta_{kj}$ for all $k>j$ and that $\|x_k^*\|=1$.
	
	Following a Gram-Schmidt argument we consider
	$y_n=z_n-\sum_{k<n} x_k^*(z_n)z_k$.
	It follows that $(y_n,x_n^*)$ is a biorthogonal sequence 
	in $M$.
	Observe also that $\rho_n(y_n)\leq n\beta_n$. We extend each $x_n^*$ to $X$ (which we keep notating $x_n^*$).
	Since $x_n^*$ has norm $1$ in $M$, this means that $|x_n^*(x)|\leq \rho_1(x)$ in $X$ and hence the $x_n^*$ are equicontinuous.
	Finally, the vectors $x_n$ will be  $x_n=\alpha(n)y_n$.
	Then, for every seminorm $\rho_j$ and $n>j$, 
	$$\rho_j(x_n)=\rho_j\l\alpha(n)y_n\r\le 
	\alpha(n)\rho_n(y_n)\leq n\beta_n \alpha(n)\to0.$$
\end{proof}

\begin{theorem}\label{Conjl1}
 Let $P\in \PP(^m \ell_1;\ell_1)$ be the polynomial $\l e_1'\r B_{\f{1}{n^4}}$ and let $X$ be any infinite dimensional and separable Fr\'echet space. Then there exists $Q\in \PP(^2X;X)$ and $\Phi\in\LL(\ell_1;X)$
 such that $Q$ is a quasiconjugacy of $P$ under $\Phi$.  In the case that $X$ is Banach we have also that $Q|_{J_Q}$ is a quasiconjugacy of $P|_{J_P}$. 
\end{theorem}
\begin{proof}
 Let $(x_n)_n\sub X$ and $(x_n^*)$ be the sequences given by the above lemma applied to the sequence $(\alpha_n)_n$ defined by 
 $\alpha_n=\f{1}{(n-1)^2}$ for $n\ge2$ and $\alpha_1=1$.
  Let $Q\in \PP(^2X;X)$ be defined as
  $$Q(x)=x_1^*(x)\sum_{n=1}^\infty \frac{x_{n+1}^*(x)x_{n}}{n^2}.$$
  Since the $x_n^*$ are equicontinuous and $x_n\to 0$ the polynomial is well defined.
  
  Let $\Phi:\ell_1\to X$ defined as $\Phi((a_n)_n)=\sum a_nx_n$. Again, since $x_n\to 0$ and $(a_n)_n\in \ell_1$ the operator is well defined.
  
  We see now that $Q$ is a quasiconjugacy of $P$ under $\Phi$. Since $span(x_n)$ is dense and it is contained in the range of $\Phi$, we have that $\Phi$ has dense range.
  
  Let $a\in \ell_1$. We have that
   $$\Phi(P(a))=\Phi(a_1 B_{\f{1}{n^4}}(a))=a_1 \sum_{n=1}^\infty \f{a_{n+1}x_n}{n^4}$$
   while 
   \begin{align*}
   Q(\Phi(a))&=Q\l\sum_{n=1}^\infty a_nx_n\r
   = x_1^*\l\sum_{n=1}^\infty a_nx_n\r \sum_{n=1}^\infty x_{n+1}^*(a_{n+1}x_{n+1}) \f{x_{n}}{n^2}\\
   &=a_1 \sum_{n=1}^\infty \f{a_{n+1}x_n}{n^4}.
   	\end{align*}

  A little more difficult is to prove that, in the Banach space case, $Q|_{J_Q}$ is a quasifactor of $P|_{J_P}$. We need to show two things.
 First that $P(J_P)\sub J_Q$ and second that $\overline {P(J_P)}=J_Q$. The first assertion is easy to prove. Let $a\in J_P$. Note that this implies that, for all $n$, $\sup_{k} \f{|P^n(a)_k|}{(k-1)^2}>1$. So we can choose $n_k$ satisfying $|P^n(a)_{n_k}|>(n_k-1)^2$. Since the $x_n^*$ are equicontinuous it follows that $\|x_n^*\|\leq C$ for some constant $C$. Therefore,
 $$C\|Q^n\Phi(a)\|=C\|\Phi P^n(a)\|\geq |x_{n_k}^*(\Phi P^n(a))|>1.$$
 Since $R_Q$ is the empty set, $\Phi(a)\in J_Q$.

 We will show now that $P(J_P)$ is dense in $J_Q$. Let $x\in J_Q$. We want to find for every $\epsilon>0$, $y\in \Phi(J_P)\cap B_\epsilon(x)$. 
 Let $T=\sum_{n=1}^\infty \frac{x_{n+1}^*(x)x_{n}}{n^2}$ so that we can write
 $Q^n(x)$ as $d_n(x)T^n(x)$ in the usual way , where $d_n(x)$ is a continuous function depending only on $\{x_1^*(x),\ldots, x_1^*(T^n(x))\}$. Notice that $c_n(a)c_n(\Omega)=d_n(\Phi(a))$. Indeed, $\Phi B_\f{1}{n^4}=T\Phi$, and we have that 
 \begin{align}\label{d_n}
 c_n(a)c_n(\Omega)T^n(\Phi(a))&=
 \Phi (c_n(a)c_n(\Omega)B^n_\f{1}{n^4}(a))= \Phi(P^n(a))\\
 \nonumber&=Q^n(\Phi(a))=d_n(\Phi(a))T^n(\Phi(a)).
 \end{align} 
 We conclude that $c_n(a)c_n(\Omega)=d_n(\Phi(a))$ for every $a\in \ell_1$.
 Without loss of generality we may suppose that there is some $t>1$ for which  $|d_n(x)|>t^{2^n}$. Therefore 
 we can find, for each $n\in\mathbb N$, $x^n\in span \{x_k:k\in \zN\}$ such that $\|x-x^n\|<\frac{\epsilon}{2}$ and $|d_{n}(x^n)|>\f{t^{2^n}}{2}$. Note that if $m_n=\max\{j:x_j^*(x^n)\neq 0\}$, then $n\leq m_n$. 
 Let $N$ such that $\sum_{j\geq N} \frac{1}{j!}\leq\frac{\epsilon}{2}$ and $\|x_n\|<1$ for every $n>N$. If each $x^n=\sum b_{j}^nx_j$ we consider the vector $a^n\in\ell_1$ 
 $$a^n_j= \begin{cases}
                   b_j^n&\text{ if }j\leq n;\\
                   sg (b_j^n)\max\{|b_{j}^n|, \frac {1}{(N-n+j)!}\}&\text{ if }j>n 
                  \end{cases}.$$
  We claim that $a^n\in J_P$ belongs to $J_P$ for large $n$. It suffices to show that $P^n(a^n)\in J_P$ for large $n$.                
  Since $\Phi(a^n)$, satisfies that $|d_{n}(\Phi(a^n))|=|d_{n}(x^n)|=|c_{n}(a^n)c_n(\Omega)|\geq \f{t^{2^n}}{2}$. We notice that $P^{n}(a^n)=c_n(a^n)c_n(\Omega)B^n_{\f{1}{j^4}}(a^n)$ has all its coordinates, with modulus greater than or equal to the coordinates of $$|c_n(a^n)c_n(\Omega)|\l\f{1}{N!n!^4},\f{2!^4}{(N+1)!(n+1)!^4},\f{3!^4}{(N+2)!(n+2)!^4},\ldots \r $$ which are  also greater than or equal to the coordinates of 
  $$\l\f{|c_n(a^n)c_n(\Omega)|}{n!^5},\f{1}{n+1!^5},\f{1}{n+2!^5},\ldots \r $$ provided that $n>N$. 
  Then by Lemma \ref{en Jp}, it follows that if $k_n$ satisfies that $$k_n\l  2^{2^{\frac{j+1}{2}}}\r^{\sqrt 2-1}\geq (j+1)!^4(j+n-1)!^5 \text{ for every } j
  	$$ 
  	then for every $L\geq k_n2^{2^{\frac{1}{2}}}$ the vector $(L,\f{1}{n!^5},\f{1}{(n+1)!^5},\ldots) \in J_P$. Note that we may take $k_n=(2n)!^5$. Thus, since $\f{t^{2^n}}{2k_n}\to \infty$, for large enough $n$, we have $\f{|c_n(a^n)c_n(\Omega)|}{n!^5}\ge k_n2^{2^{\frac{1}{2}}}$ and therefore  the vector $P^n(a^n)$ (and hence $a^n$) is eventually in $J_P$. Consequently, $\Phi(a^n)\in J_Q$.               
 
 Finally we have that
\begin{align*}
 \|\Phi(a^n)-x\|&\leq \|x-x^n\|+\|\Phi(a^n)-x^n\|\\
 &\leq \frac{\epsilon}{2}+\sum_{j=n+1}^{m_n}  (b_j^n-a_j^n)  \|x_j\| + \sum_{j=m_n+1}^\infty \frac{1}{j!}\|x_j\|\\
 &\leq \frac{\epsilon}{2}+ \left\|\l\frac{1}{j!}\r_{j\geq N}\right\|_{\ell_1}\leq \epsilon.
\end{align*}
 \end{proof}
\begin{theorem}\label{main theorem}
Let $X$ be an infinite dimensional separable Fr\'echet space, then there exists a non linear homogeneous polynomial that is at the same time weakly hypercyclic, $d$-hypercyclic and $\Gamma$-supercyclic for each $\Gamma\subset \mathbb C$ which is unbounded or has zero as an accumulation point.
\end{theorem}
\begin{proof}
	By Theorem \ref{Conjl1} there exists $Q\in \PP(^2X;X)$ a quasiconjugacy of $P=e_1'(x)B_{\f{1}{n^4}}$ under a linear factor $\Phi$. This polynomial is, by Theorem \ref{Ejl1} weakly hypercyclic, $d$-hypercyclic and $\Gamma$-supercyclic for every unbounded or not bounded away from zero $\Gamma$. Since $\Phi$ is a linear operator, it follows by Propositions \ref{d hiper}, \ref{weakly conj} and \ref{conjugar Gamma} that $Q$ satisfies the required properties.
\end{proof}

\end{document}